\documentclass[final,onefignum,onetabnum]{siamart190516}

\usepackage{lipsum}
\usepackage{clrscode}
\usepackage{appendix}
\usepackage{bm}
\usepackage{booktabs}
\usepackage{url}
\usepackage{multirow}
\usepackage{textcomp}
\usepackage{amsmath}
\usepackage{amsfonts}
\usepackage{amssymb}
\usepackage{mathrsfs}
\usepackage{hyperref}
\usepackage{graphicx,graphics,subfigure}
\usepackage{hyperref,url}
\usepackage{epsf,epstopdf}
\usepackage{algorithm}
\usepackage{algorithmic}
\usepackage{essay-def-siam}
\usepackage{bbm}
\usepackage{dsfont}

\usepackage{indentfirst}
\usepackage{pdfpages}
\usepackage{pdflscape}
\usepackage{longtable}
\ifpdf
  \DeclareGraphicsExtensions{.eps,.pdf,.png,.jpg}
\else
  \DeclareGraphicsExtensions{.eps}
\fi

\newcommand{\stepa}[1]{\overset{\rm (a)}{#1}}
\newcommand{\stepb}[1]{\overset{\rm (b)}{#1}}
\newcommand{\stepc}[1]{\overset{\rm (c)}{#1}}
\newcommand{\stepd}[1]{\overset{\rm (d)}{#1}}

\headers{A DECOMPOSITION ALM FOR LOW-RANK SDP}{Y. Wang, K. Deng, H. Liu and Z. Wen}

\title{A Decomposition Augmented Lagrangian Method for Low-rank Semidefinite Programming\thanks{Submitted to the editors DATE.
\funding{Z. Wen was supported in part by the NSFC grant 11831002.}}}

\author{Yifei Wang\thanks{Department of Electrical Engineering, Stanford University, United States
  ({wangyf18@stanford.edu}).}
  \and  Kangkang Deng$^{\ddagger}$ \and Haoyang Liu$^{\ddagger}$  \and
Zaiwen Wen\thanks{Beijing International Center for Mathematical Research, Peking
University, China  ({\{dengkangkang,liuhaoyang,wenzw\}@pku.edu.cn}).}
}

\usepackage{amsopn}

\ifpdf
\hypersetup{
  pdftitle={A Decomposition ALM for Low-rank SDP},
  pdfauthor={Y. Wang, K. Deng, H. Liu, and Z. Wen}
}
\fi

\newcommand{\revise}[1]{{#1}}
\newcommand{\blue}[1]{{#1}}

\begin{document}

\maketitle

\begin{abstract}
We develop a decomposition method based on the augmented Lagrangian framework to solve a broad family of semidefinite programming problems\revise{,} possibly with nonlinear objective functions, nonsmooth regularization, and general linear equality/inequality constraints.  In particular, the positive semidefinite variable along with a group of linear constraints can be transformed into a variable on a smooth manifold \revise{via matrix factorization}. The nonsmooth regularization and other general linear constraints are handled by the augmented Lagrangian method. Therefore, each subproblem can be solved by a semismooth Newton method on a manifold. Theoretically, we show that the first and second-order necessary optimality conditions for the factorized subproblem are also sufficient for the original subproblem under certain conditions. Convergence analysis is established for the Riemannian subproblem and the augmented Lagrangian method. Extensive numerical experiments on large-scale semidefinite programming \revise{problems} such as max-cut, nearest correlation estimation, clustering, and sparse principal component analysis demonstrate the strength of our proposed method compared to other state-of-the-art methods. 
\end{abstract}



\section{Introduction}

Let $\mbS^n$ denote the linear space of symmetric matrices with size $n\times n$.
\revise{Let $A_1,\cdots,A_m,B_1,\cdots,B_{m_0}\in\mbS^n$ be a given set of matrices.}    This paper aims to solve a general composite semidefinite programming problem (SDP):
\begin{equation}\label{sdp:general}
\begin{aligned}
\min_{X\in \mcD}& \quad  f(X)+h(X),\quad \text{s.t.}  \quad\mcA(X)=b,
\end{aligned}
\end{equation}
where 
$ \mcA(X)=[\tr( A_1 X),\dots, \tr( A_{ m} X)]^T$.
The function $f(X)$ is smooth while $h(X)$ is possibly nonsmooth (see the details
in Assumption \ref{assumption1}). The domain $\mcD=\{X\in \mbS^n|X\succeq 0, \mcB(X)=b_0\}$ with $ \mcB(X)=[\tr( B_1 X),\dots, \tr( B_{ m_0} X)]^T$ defines \revise{a} certain Riemannian structure. 
We can recover the linear SDP by taking $f(X)=\tr(CX)$, $h(X)=0$ and $\mcD=\{X\succeq 0\}$ where $C\in \mbS^n$. 
The linear SDP is of central importance in convex optimization. 
It serves as a tractable \revise{convex relaxation \cite{boyd2004convex, wolkowicz2012handbook, anjos2011handbook}} of many (possibly NP-hard) difficult problems from combinatorial optimization, constraint \revise{satisfaction}, computer vision, machine learning, etc.


\revise{Our} reason for considering SDP in the general form of \eqref{sdp:general} is that it covers a broader range of optimization problems with semidefinite variables, especially for the problems with \revise{a} nonlinear function $f(X)$ and/or \revise{a} nonsmooth regularization $h(X)$. For example, \revise{the} nonlinear $f(X)$ appears in \revise{the} nearest correlation problem 
\cite{qsdpnal} and \revise{a} convex formulation of \revise{the} neural network training problem
\cite{bartan2021neural}. The nonsmooth regularization term $h(X)$ usually
improves the \revise{ solution quality } for the original problem. In sparse PCA, $h(X)$ is the $\ell_1$ norm to improve \revise{sparsity}. In the
Lov{\'a}sz theta problem and clustering problems, $h(X)$ can be the indicator function of the set of
non-negative matrices $\mcX = \{X\in \mbS^n|X\geq 0\}$. The non-negative
constraints on the semidefinite variable can be further \revise{strengthened to} the
non-negative constraints on the factorized variable. \revise{Also}, $h(X)$ can be the entropic penalty function
\cite{ijcai2019-157} 
to find a low-rank solution. 

Our approach developed for \eqref{sdp:general} can also deal with the following two alternative formulations:
\begin{align}
\min_{X\in \mcD}& \quad  f(X)+h(X),\quad \text{s.t.}  \quad\mcA(X)=b,\; \mathrm{rank}(X) \leq r, \label{prob:rank-constr}\\
\min_{X\in \mcD}& \quad f(X) + h(R),\quad \text{s.t.}\quad \mcA(X) = b, \; X = R^TR. \label{prob:hR-fact}
\end{align}
\revise{These problems} play an important role in low-rank optimization and combinatorial
optimization including the max-cut
problem and the theta problem in section
\ref{sec:num}. 
\revise{The storage of the variables $X$ and  $R$ is discussed in Section \ref{sec:extension}.}

\subsection{Literature review}
Due to the extensive applications of SDP, it is of great interest to develop efficient algorithms to solve large-scale SDPs, especially those with complex constraints and general objective functions. The class of interior-point methods contains the most popular polynomial-time algorithms for solving linear SDP with small to medium scale problems. Nevertheless, interior point methods are not very efficient for sparse or large-scale problems due to the computational cost of second-order search directions. A detailed description of interior-point methods for solving SDP can be found in \cite{hosdp}. \revise{The paradigm based on the augmented Lagrangian framework has also been extensively studied.  Yang et al. \cite{sdpnalp} propose a majorized semismooth Newton-CG augmented Lagrangian method, called SDPNAL+, for SDP problems with nonnegative constraints.   A two-phase augmented Lagrangian method, called QSDPNAL,  is developed by Li et al. \cite{qsdpnal} to solve the convex quadratic SDP problem. 
}

The major challenge in solving the linear SDP is the constraint $X\succeq 0$. To circumvent this constraint, Burer and Monteiro \cite{sdplr} recast the linear SDP by factorizing $X=R^TR$ with $R\in \mbR^{p\times n}$.
Then, they apply the augmented Lagrangian method (ALM) and use \revise{a} quasi-Newton method to solve the factorized subproblem. 
In theory, the linear SDP has a global optimum with rank at most $p$, where \revise{ $p$ is the largest integer such that} $\frac{p(p+1)}{2}\leq \revise{m+m_0}$, see \cite{otroe,podga}. Based on this observation, a global solution of the \revise{linear} SDP can be found from the factorized problem with a suitable rank $p$. 
 Sahin et al. \cite{aialf} propose an inexact ALM framework and establish the convergence to \revise{a} first (or second) order optimum of the factorized problem by solving the ALM subproblem with a first (or second) order approximate stationary point. 

Following the non-convex Burer-Monteiro approach, Journ\'ee et al. \cite{lroot} and Boumal et al. \cite{tncbm} reformulate the linear SDP by leaving all linear constraints to $\mcD$. They show that for $\mcB$ satisfying certain conditions, the set $\mcM=\{R\in \mbR^{p\times n}: \mcB(R^TR)= b_0\}$ can be viewed as a smooth manifold. 
Journ\'ee et al. \cite{lroot} 
provide a sufficient condition to ensure that $\mcM$ is a smooth manifold. 
In this case, the factorized problem becomes a Riemannian optimization problem.
If $R$ is a local optimum and $R$ is rank-deficient, then $R$ is a global optimum of this Riemannian optimization problem. 
To solve this Riemannian optimization problem, Journe\'e et al. apply the
Riemannian trust-region method \cite{oaomm}. 
Under a milder assumption on the smooth
manifold $\mcM$, Boumal et al. \cite{tncbm} show that for $\frac{p(p+1)}{2}> \revise{m+m_0}$,
for almost all $C\in \mbS^n$, any second-order critical point $R$ of the
Riemannian optimization problem is globally optimal and $X=R^TR$ is global
optimal to the linear SDP. Similarly, they apply the Riemannian trust-region
method and provide a global convergence rate estimation by applying convergence
results in \cite{grocf}. 

The idea of factorizing $X=R^TR$ to bypass the semidefinite constraint is also presented in the following works. Shah et al. \cite{brfsp} consider the case where $C$ and $A_i$ are all positive semidefinite matrices. Then, the factorized problem can be \revise{relaxed}  into a convex optimization problem. They also propose a general initialization scheme to start close to a global optimum.
Bhojanapalli et al. \cite{dcffs} focus on the convex SDP without constraints and apply the gradient descent method on the factorized problem with a carefully designed step size. The non-convexity of the factorized problem is partially overcome via some initialization techniques. 

\subsection{Our contribution}
         We develop a decomposition-based augmented Lagrangian framework, named SDPDAL\revise{, which} can solve general semidefinite programming problems with nonlinear objective functions, nonsmooth regularizations, and general linear constraints. Decomposition-based methods \cite{sdplr,lroot,brfsp} \revise{have} primarily \revise{focused} on solving linear SDPs\revise{,} and they 
    cannot easily be extended to \revise{tackle} \eqref{sdp:general}. Our proposed method retains some of the linear constraints and deals with the remaining constraints using ALM and variable splitting. We factorize $X = R^TR$ for the ALM subproblem whose gradient is semi-smooth, and the retained linear constraints form a Riemannian manifold. Then, an adaptive regularized semismooth Newton method on manifold is developed.  The reason for introducing the manifold is to take full advantage of the problem structure and reduce the numerical difficulty during the update of multipliers. 
    Different from SDPLR \cite{sdplr}, we do not apply the L-BFGS method to solve the subproblem since some of the constraints are preserved as a manifold and \text{the objective function has a nonsmooth term}. 
    Our method is also distinguished from \cite{lroot} in the sense that we are able to deal with general constraints which may be hard to reformulate into a manifold.
    Consequently, SDPDAL extends the types of SDPs that a decomposition-based method can solve, while it inherits the advantage on large-scale low-rank problems.

Properties of the factorized subproblem with respect to $R$ from ALM are theoretically
analyzed \blue{and they provide interesting insights into the convergence and
    efficient implementation of SDPDAL}. We show that \revise{a} rank-deficient local minimum is also globally
optimal, \revise{and in} a small neighborhood of the global optima, any critical
point is shown to be globally optimal. We also bound the optimality gap of the
subproblem by the gradient norm and the smallest eigenvalue of a dual variable
$S\in \mbS^n$. \blue{Although our algorithm does not guarantee that $S$ will
become positive semidefinite, our numerical implemantion shows that the negative
eigenvalues of $S$ will vanish. If the initial point $R_0$ is assumed to be  close to
 an optimal solution $R^*$, we show that the optimality gap will be bounded by
 the gradient norm. A complexity analysis and an asymptotic convergence rate
 of the semismooth Newton method are 
 established for the factorized problem  under certain mild assumptions. In
 particular, we
 bound the maximal number of iterations to reach a first-order critical point
 and prove the global convergence.  
     A few numerical strategies,  such as
 saddle points escaping and initialization  using  first-order algorithms, 
are also proposed to meet those assumptions (see the details in \cite{WangDengLiuWen2021arxiv}).}
 Finally, the convergence to the KKT pairs of the factorized problem \blue{is fit
 together} and
 the convergence to the optimal solution for the original SDP are established
 under certain mild conditions.  
        For \revise{the} convex \revise{case of} \eqref{sdp:general}, we
        formulate its dual and construct \revise{values for the} dual variables
        from the primal variables. Therefore, stopping criteria can be defined
        based on the KKT conditions so that a robust implementation of SDPDAL is
        possible for solving the primal problem. On max-cut \revise{problem},
        theta problem, clustering and sparse principal component analysis,
        SDPDAL outperforms several state-of-the-art solvers, especially
        \revise{for} problems where the dimension of $X$ is larger than 5000.
\revise{
\subsection{Notation} Given a matrix $A$, we use $\|A\|_F:=\sqrt{ \sum_{ij}A_{ij}^2}$ to denote its Frobenius norm, $\|A\|_1:=\sum_{ij}|A_{ij}|$ to denote its $\ell_1$ norm, and $\|A\|_2: = \sigma_{\max}(A)$ to denote its spectral norm, where $\sigma_{\max}(A)$ represents the largest singular value of matrix $A$. For a linear operator $\mcA$, we use $\|\mcA\|_{\text{op}}$ to represent its operator norm. For a vector $x$, we use $\|x\|_2$ and $\|x\|_1$ to denote its Euclidean norm and $\ell_1$ norm, respectively.
We define $\lra{A,B} = \tr(A^TB)$ 
as the inner product of any $A, B \in \mathbb{R}^{m\times n}$.
} 
\subsection{Organization}
This paper is organized as follows. In section \ref{sec:alm}, we introduce the
ALM framework, apply the low-rank factorization to the ALM subproblem and
develop an adaptive semismooth Newton method. A theoretical analysis
of the subproblem properties and convergence analysis for both the subproblem
and ALM are provided in section \ref{sec:thm}. 
Numerical experiments are presented in section \ref{sec:num}.

\section{Augmented Lagrangian method with low-rank factorization}\label{sec:alm}
In this section, we present a framework of the augmented Lagrangian method (ALM) for SDP with the form \eqref{sdp:general}. 
To solve the subproblem in ALM, we propose an adaptive regularized Riemannian Newton method.
\subsection{Formulation of the optimization problem}
For the linear constraints $\mcA(X)=b$ and $\mcB(X)=b_0$ in \eqref{sdp:general}, $\mcA$ can be an arbitrary linear operator while $\mcB$ shall have certain structures.  
In this paper, we assume that the following statement holds.
\begin{assumption}\label{assumption1}
\text{ }
\begin{enumerate}[label={\textbf{\Alph*:}},
  ref={\theassumption.\Alph*}]
  \item \label{assump-B}%
    The problem \eqref{sdp:general} has a low-rank optimal solution. The
    objective function is bounded from below. The function $f(X)$ is convex and
    twice-differentiable and the gradient $\nabla f(X)$ is Lipschitz continuous, while $h(X)$ is a convex, lower semicontinuous, and proper function.
\item \label{assump-A}
    For $X\in \mcD$, $\{B_iX \}_{i=1}^{m_0}$ are linear independent. $\mcM=\{R\in \mbR^{p\times n}|R^TR\in \mcD\}$ is a compact submanifold embedded in the Euclidean space $\mbR^{p\times n}$.
\end{enumerate}
\end{assumption}
Typical examples of the domain $\mcD$ include
\begin{equation}\label{equ:d}
\begin{aligned}
&\mcD = \{X \mid \tr(X)=1,\; X\succeq 0\}, \\
&\mcD = \{X\mid X_{i,i}=1,\; X\succeq 0\}, \\
&\mcD = \{X \mid X_{i_{j-1}+1:i_{j}, i_{j-1}+1:i_{j}}=I_{d_j}, 1\leq j\leq q,\; X\succeq 0\}.
\end{aligned}
\end{equation}
For the last example, $0=i_0<i_1<\dots<i_q=n$ and $d_j=i_{j}-i_{j-1}$ for $1\leq j\leq q$. We factorize $X=R^TR$ where $R\in \mbR^{p\times n}$, and denote
\begin{equation}\label{equ:m}
\mcM=\{R\in \mbR^{p\times n}|R^TR\in \mcD\}=\{R\in \mbR^{p\times n}|\mcB(R^TR)=b_0\}.
\end{equation}
\blue{Assumption
\ref{assump-A} ensures that $\mcM$ is a smooth manifold, see \cite[Assumption 1.1]{boumal2020deterministic}.} For examples of $\mcD$ in \eqref{equ:d}, $\mcM$ corresponds to different smooth manifolds:
\begin{equation}\label{specified example}
\begin{aligned}
&\mcM = \{R\in \mbR^{p\times n}| \|R\|_F=1\};\\
&\mcM = \{R\in \mbR^{p\times n}| R=[r_1,\dots, r_n],\; \|r_i\|_2=1\};\\
&\mcM = \{R\in \mbR^{p\times n}| R=[R_1,\dots, R_q],\; R_j\in \mbR^{p\times d_j},\; R_j^TR_j=I_{d_j}\},
\end{aligned}
\end{equation}
where $\|R\|_F$ represents the Frobenius norm of $R$. The relationship between the SDP
problem respect to $X$ and the factorized problem respect to $R$ will be
mentioned in Proposition \ref{propi:connection}. We note that all examples in \eqref{equ:d} satisfy Assumption \ref{assump-A}. 


\subsection{An augmented Lagrangian method based on splitting}
The problem \eqref{sdp:general} can be rewritten as
\begin{equation}\label{prob:p2}
\begin{aligned}
\min_{X\in \mcD, W\in \mbS^n} \quad  f(X)+h(W),\quad \text{s.t.} \quad \mcA (X)=b, \quad X=W.
\end{aligned}
\end{equation}
We apply ALM to solve \eqref{prob:p2}. Denote the augmented Lagrangian function associated with \eqref{prob:p2} by
\begin{equation}\label{aug_L:1}
\begin{aligned}
L_\sigma(X,W,y,Z) = &f(X)+h(W)-y^T(\mcA(X)-b)-\lra{Z,X-W}\\
&+\frac{\sigma}{2}\pp{\|\mcA(X)-b\|_2^2+\|X-W\|_F^2},
\end{aligned}
\end{equation}
where $y\in \mbR^m,Z\in \mbS^n$ and $\sigma>0$ is a parameter for ALM. For $A,B\in \mbR^{n\times n}$, $\lra{A,B}=\tr(A^TB)$ represents the Euclidean inner product in $\mbR^{n\times n}$. The domain of the primal variables $(X,W)$ is $X\in \mcD$ and $W\in \mbS^n$. The $k$-th iteration of the ALM is given as follows:
\begin{align}
  (X^{k+1}, W^{k+1}) &= \mathop{\argmin}_{X\in \mcD, W\in \mbS^n} L_{\sigma_k}(X,W,y^{k},Z^{k}),\label{sub:xw} \\
  y^{k+1} &=y^k-\alpha_k\sigma_k(\mcA (X^{k+1}) -b),\notag\\
  Z^{k+1} &=Z^k-\alpha_k\sigma_k(X^{k+1} -W^{k+1}).\notag
\end{align}
For a fixed $X$, the optimal solution of $W$ to \eqref{sub:xw} follows
\begin{equation}\label{equ:w}
W = \prox_{h/\sigma_k}(X-Z^k/{\sigma_k}),
\end{equation}
where $\prox_{h}(W)$ is the proximal mapping of a convex function $h$ defined by
$
\prox_{h}(W)=\argmin_{V\in \mbS^n} h(V)+\frac{1}{2}\|V-W\|_F^2.
$
Denote
\[
\begin{aligned}
  \Phi_k(X)  :=& \inf_{W\in\mbS^n} L_{\sigma_k}(X,W,y^{k},Z^{k})
  = f(X)+h(\prox_{h/\sigma_k}(X-Z^k/{\sigma_k})) - \frac{1}{2\sigma_k}\|Z^k\|_F^2\\
            &  +\frac{{\sigma_k}}{2}\Big(\|\mcA(X)-b-y^k/{\sigma_k}\|_2^2 +\|
            X-Z^k/{\sigma_k}-\prox_{h/\sigma_k}(X-Z^k/{\sigma_k})
            \|_F^2\Big).\\
\end{aligned}
\]
Then, the optimal $(X^{k+1}, W^{k+1})$ for \eqref{sub:xw} can be computed as follows:
\[
X^{k+1} = \argmin_{X\in \mcD}\Phi_k(X),~~~~~W^{k+1} = \prox_{h/\sigma_k}(X^{k+1}-Z^k/{\sigma_k}).
\]
According to the Moreau decomposition, we have
$ X - \prox_{th}(X) =t\prox_{h^*/t}(X/t), $
where $h^*$ is the conjugate function of $h$ defined by
$ h^*(W) = \sup_{V\in \mbS^n} \{\lra{U,V}-h(V)\}$.
Denote $T(X)=\sigma_k^{-1}\prox_{\sigma_k h^*}(\sigma_kX-Z^k)$. By ignoring the constant term, we can rewrite the minimization problem of $X$ as
\begin{equation}\label{sub:xk}
\begin{aligned}
\min_{X\in \mcD}\Phi_k(X) =& f(X)+h(X-Z^k/{\sigma_k}-T(X))\\
&+\frac{{\sigma_k}}{2}\Big(\|\mcA(X)-b-y^k/{\sigma_k}\|_2^2+\|T(X)\|_F^2\Big).
\end{aligned}
\end{equation}

To solve the subproblem \eqref{sub:xk}, we factorize $X=R^TR$. Instead of directly minimizing \eqref{sub:xk}, we consider the following Riemannian optimization problem:
\begin{equation}\label{sub:rk}
\min_{R\in \mcM} \Psi_k(R):=\Phi_k(R^TR).
\end{equation}
Because $\Phi_k(X)$ is continuously differentiable but may not be twice continuously differentiable, we apply an adaptive regularized semismooth Newton method to \eqref{sub:rk}\revise{. This is} discussed with details in subsection \ref{ssec:ssn}. Suppose that $R^{k+1}$ is an approximate solution to \eqref{sub:rk}. Then, we use $X^{k+1}=(R^{k+1})^TR^{k+1}$ as an approximate solution to \eqref{sub:xk}. 
The overall algorithm is summarized in Algorithm \ref{alg:alm_ssn}.

\revise{
We emphasize that computing $X = R^TR$ explicitly is often not required when only a small number of elements $X_{ij}$ is needed or the operations can be performed on $R$ directly. For example, i) $f(X)$ is linear,
ii) $A_1,\ldots,A_m$ are sparse or low-rank and iii) $h(X) = 0$.
By utilizing the fact that $\lra{C,R^TR} = \lra{RC,R}$, we are able to
compute $f(R^TR)$, $\mathcal{A}(R^TR)$, and $R\nabla f(R^TR)$ directly via
 matrix multiplications. However, when elementwise operations on $X$ are involved,  we need to form $X = R^TR$ explicitly.
}

\begin{algorithm}[htbp]
\caption{The SDPDAL method (prototype)}\label{alg:alm_ssn}
\begin{algorithmic}[1]
\REQUIRE Initial trial point $R^0\in \mcM$ 
, ALM step size $\alpha_k$, parameters $\sigma_k>0$.
\STATE Set $k=0, y^k=0,Z^k=0$. 
\WHILE {not converge}
\STATE Obtain $R^{k+1}$ by solving \eqref{sub:rk} inexactly. Formulate $X^{k+1}=(R^{k+1})^TR^{k+1}$ \revise{either explicitly or implicitly}.
\STATE Update 
$
W^{k+1} = \prox_{h/\sigma_k}(X^{k+1}-Z^k/{\sigma_k}).
$
\STATE Update Lagrangian multipliers $y^{k+1}$, $Z^{k+1}$ by
{\setlength\abovedisplayskip{5pt}
\setlength\belowdisplayskip{-5pt}
$$
\begin{aligned}
y^{k+1} = &y^k-\alpha_k\sigma_k(\mcA(X^{k+1})-b),\;
Z^{k+1} = Z^k-\alpha_k\sigma_k(X^{k+1}-W^{k+1}).
\end{aligned}
$$}
\STATE Update $\sigma_{k+1}$ \revise{and $\alpha_{k+1}$}.

 Set $k=k+1$.
\ENDWHILE
\end{algorithmic}
\end{algorithm}

\subsection{Calculation details}

To apply the adaptive regularized Riemannian Newton method, we first \revise{introduce the following definitions to} characterize the generalized Hessian of $\Phi_k(X)$.
\begin{definition}
Let $\mcO\subseteq \mbR^n$ be an open set and $P:\mcO\to \mbR^m$ be locally Lipschitz continuous at $x\in \mcO$. Denote by $\revise{D(P)}$ the set of differentiable points of $P$ in $\mcO$. The B-subdifferential of $P$ at $x$ is defined by
$$
\p_B P(x):=\left\{\lim_{k\to \infty}P'(x^k)|x^k\in \revise{D(P)},x^k\to x\right\}. 
$$
The set $\p P(x)=\operatorname{co}(\p_B P(x))$ is called Clarke differential of $P$ at $x$, where $\operatorname{co}$ denotes the convex hull.
\end{definition}
\revise{
\begin{definition}
A locally Lipschitz continuous operator $P$ is called semismooth at $x$ if
 $P$ is directional differentiable at $x$,  and for all $d$ and $J \in \partial P(x+d)$, it holds that
\[ \| P(x+d) -  P(x) - Jd \|_2 = o(\|d\|_2), \;\; d \rightarrow 0. \]
We say $P$ is semismooth if $P$ is semismooth for any $x \in \mbR^m$.
\end{definition}
}

We note that $h(X-Z^k/{\sigma_k}-T(X))+\frac{\sigma_k}{2}\|T(X)\|^2_2$ is the Moreau envelop function, and its gradient is simply $\sigma_kT(X)$.
Hence, the Euclidean gradient of $\Psi_k(R)$ is
\[
\begin{aligned}
\nabla  \Psi_k(R) 
=& 2R(\nabla f(R^TR)+{\sigma_k} \mcA^*(\mcA(R^TR)-b-y^k/{\sigma_k})+{\sigma_k} T (R^TR)).
\end{aligned}
\]
According to the property of the proximal mapping, $T(X)$ is strongly semismooth \cite{smvf}. Via the Clarke differential, the following operator is well-defined:
\begin{equation}
\begin{aligned}
\hat \p^2 \Psi_k(R)[U]
=& 2U(\nabla f(R^TR)+{\sigma_k} \mcA^*(\mcA(R^TR)-b-y^k/{\sigma_k}))\\
&+2{\sigma_k} R\mcA^*\mcA(R^TU+U^TR)+2R\nabla^2 f(R^TR)[R^TU+U^TR]\\
&+2{\sigma_k} U T(R^TR)+2\sigma_k R \p T(R^TR)[R^TU+U^TR].
\end{aligned}
\end{equation}
From \cite{ghmas}, the generalized Hessian operator $\p^2 \Psi_k(R)$
can be expressed as
$
\p^2 \Psi_k(R)[U]=\hat \p^2 \Psi_k(R)[U].
$
Then the Riemannian gradient writes
\[\grad  \Psi_k(R) = \mfP_{T_R\mcM}(\nabla \Psi_k(R)),
\]
 where \revise{ $T_R\mcM$ is the tangent space at $R\in \mcM$ defined as $T_R\mcM = \{U~|~\mcB(R^TU + U^TR) = 0\}$ and}  $\mfP_{T_R\mcM}$ is the projection into the tangent space at $R\in \mcM$. On the other hand, 
the generalized Riemannian Hessian satisfies
\begin{equation}\label{equ:r_hess}
\Hess  \Psi_k(R)[U] = \nabla_U \grad \Psi_k(R) =   \mfP_{T_R\mcM}(\p^2  \Psi_k(R)[U])+\mathfrak{W}_R(U,\mfP^\perp_{T_R\mcM} (\nabla \Psi_k(R))),
\end{equation}
where $\mathfrak{W}_R(U,V) = -\mfP_{T_R\mcM}D_U V$ with $U\in T_R\mathcal{M}, V\in T_R^\perp\mcM$ is a symmetric linear operator, \revise{and $D_U(V) = \lim_{t\rightarrow 0}V(\gamma(t))$, where $\gamma$ is any curve on $\mcM$ with $\gamma(0) = R, \gamma'(0) = U$. This}  is related to the second fundamental form of $\mcM$. Detailed definitions can be found in \cite{aelat}.

Denote $\mcB^*:\mbR^{m_0}\to \mbS^n$ the adjoint operator of $\mcB$ defined by $\mcB^*(v)=\sum_{i=1}^{m_0} B_iv_i$. \revise{By Assumption \ref{assump-A},} let $u(X)$ be the unique solution to 
\begin{equation}\label{equ:u}
\mcB(X(\nabla\Phi_k(X)-\mcB^*(u)))=0
\end{equation}
and denote
\begin{equation}\label{equ:S}
S(X)=\nabla\Phi_k(X)-\mcB^*(u(X)).
\end{equation}
When there is no confusion, we omit the variable in bracket and use $u,S$ to
represent $u(X),S(X)$ or $u(R^TR),S(R^TR)$, respectively. \revise{Note that $u,X$ are corresponding to the Lagrangian multiples of the constraints $\mcB(X) = b_0$ and $X\succeq 0$, respectively. We will further discuss it in the dual formulation \eqref{prob:d}.} We present the
detailed calculation of the Riemannian gradient and generalized Riemannian Hessian in the following proposition. 
 The proof is omitted due to page limit and can be found in
 \cite{WangDengLiuWen2021arxiv}.
\begin{proposition}\label{prop:rgrad}
1) The Riemannian gradient satisfies
\begin{equation}\label{equ:r_grad}
    \grad  \Psi_k(R)=\nabla  \Psi_k(R)-2\sum_{i=1}^{m_0}u_iRB_i,
\end{equation}
where $u_i$ denotes the $i$-th entry of $u(R^TR)$.
2) For $U\in T_R\mcM$, the generalized Riemannian Hessian is
\begin{equation}\label{equ:r_hess_c}
\Hess \Psi_k(R)[U]  = 2U\nabla \Phi_k(R^TR) +2R\hat \p^2 \Phi_k(R^TR)[U^TR+R^TU]-2R\mcB^*(u')-2U\mcB^*(u),
\end{equation}
where $u\in \mbR^{m_0}$ denote $u(R^TR)$ and $u'\in \mbR^{m_0}$ satisfies
\[
\mcB\pp{R^TR\pp{\hat\p^2  \Phi_k(R^TR)[U^TR+R^TU]-\mcB^*(u')}}=0.
\]
\revise{ By \eqref{equ:S}, we further have
\begin{equation}\label{equ:hess_prod}
    \lra{U,\Hess  \Psi_k(R)[U]}  = 2\tr(USU^T)+\lra{U^TR+R^TU, \hat \p^2 \Phi_k(X)[U^TR+R^TU]}.
\end{equation}}
\end{proposition}

\begin{remark}
When $B_iB_j=0$ for $1\leq i<j\leq m_0$, then $u$ has a closed-form formula:
$ u_i = \frac{\tr(X\nabla \Phi_k(X)B_i)}{\tr(B_iXB_i)}$.
\end{remark}

\subsection{A semismooth Newton method on manifold}\label{ssec:ssn}
Motivated by the Adaptive Regularized Newton Method for Riemannian Optimization (ARNT) \cite{arnt}, we introduce an adaptive regularized Riemannian semismooth Newton method for solving \eqref{sub:rk} to a high-precision.
\revise{At a point $R_l$,} we consider the following subproblem
%
\begin{equation}\label{sub:mk}
\min_{U\in T_{R_l}\mcM}  m_l(U)=\lra{\grad \Psi_k(R_l),U}+\frac{1}{2} \lra{ \mfH_l[U],U}+\frac{\nu_l}{2} \|U\|_F^2,
\end{equation}
where $\mfH_l\in \Hess \Psi_k(R_l)$ is a generalized Riemannian Hessian operator and $\nu_l>0$ is a regularization parameter. 
The construction of the subproblem \eqref{sub:mk} in the tangent vector space is different from that in ARNT.

In each step, we inexactly solve the following linear equation:
\begin{equation}\label{sub:mk_linear}
\grad \Psi_k(R_l)+\mfH_l[U]+\nu_l U=0,
\end{equation}
where the  step $U_l\in T_{R_l} \mathcal{M}$ satisfies the first-order condition:
\begin{equation}\label{condition}
  \|\nabla m_l(U_l)\|_{\revise{F}} \leq \theta \|U_l\|_{\revise{F}},~~ \revise{m_l(U_l)\leq 0,}
\end{equation}
\revise{where $\theta\geq 0$ is a constant. The above condition has been used in manifold optimization  \cite{agarwal2021adaptive}.} We apply the modified conjugate gradient (mCG) method in \cite{arnt} to solve the problem  \eqref{sub:mk_linear}. Since the Riemannian generalized Hessian operator $\mfH_l$ may not be positive definite, we also terminate the CG method when a negative or small curvature is encountered. 
Then, we construct a gradient-related direction based on conjugated directions and perform a curvilinear search in this direction to ensure that the output $U_l$ of the mCG algorithm is a descent direction, i.e.,
$ \lra{U_l, \grad \Psi_k(R_l)}<0. $
This is justified in Lemma 7 in \cite{arnt}. Once an approximate solution $U_l$ of \eqref{sub:mk_linear} is obtained, we perform a line search along $U_l$ to generate a trial point
\[ \bar R_l = \mcR_{R_l}(\revise{s_l}U_l),\]where $\mcR$ is a retraction operator on $\mcM$. Here the step size $\revise{s_l}=\delta^h$ is chosen by the Armijo condition such that $h$ is the smallest nonnegative integer satisfying
$ m_l(\delta^hU_l)\leq \mu\delta^h \lra{\nabla m_l(0), U_l}, $
where $\mu>0$ is a parameter for the line search.

Let $0<\eta_1\leq \eta_2<1$ be the parameters. To decide whether to accept $\bar R_l$ or not, we compute the following ratio between the actual reduction and the predicted reduction
\begin{equation}\label{equ:rhol}
\rho_l = \frac{\Psi_k(\bar R_l)- \Psi_k(R_l)}{m_l(U_l) - \frac{\nu_l}{2}\|\revise{U_l}\|_F^2}.
\end{equation}
If $\rho_l\geq \eta_1>0$, then the iteration is successful and we set $R_{l+1}=\bar R_l$. Otherwise the iteration fails and we set $R_{l+1} = R_l$. The regularization parameter $\nu_{l+1}$ is updated as follows
\begin{equation}
 \label{alg:tau-up} \nu_{l+1} \in \begin{cases} [\max (\nu_{\min},\gamma_0\nu_l),\nu_l], & \mbox{  if }
	\rho_l \ge \eta_2, \\ [\nu_l,\gamma_1 \nu_l], & \mbox{  if } \eta_1 \leq \rho_l
	< \eta_2, \\ [\gamma_1 \nu_l, \gamma_2 \nu_l], & \mbox{  otherwise}, \end{cases}
\end{equation}
where $0<\gamma_0<1<\gamma_1\leq \gamma_2$ \revise{ and $ \nu_{\min}$} are parameters. The regularized semismooth Newton method to solve the $R$-subproblem \eqref{sub:rk} is summarized in Algorithm \ref{alg:ssn_r}.

\begin{algorithm}[ht]
\caption{Semismooth Newton method for \eqref{sub:mk}}\label{alg:ssn_r}
\begin{algorithmic}[1]
\STATE Input: an initial point $R_0\in \mcM$, $\nu_0>0$, a maximum iteration number \revise{ $K$}. 
 Choose $0<\eta_1\leq \eta_2<1,0<\gamma_0<1<\gamma_1\leq \gamma_2$.
 Set $l=0$.
\WHILE {$l<\revise{K} $ \revise{ \text{ and not converge}}}
\STATE Compute a new trial point $\bar R_l$ by solving \eqref{sub:mk} with mCG.
\STATE Compute the ratio $\rho_l$ via \eqref{equ:rhol}.
\STATE Update $R_{l+1}=  \bar R_l$ if $\rho_l\geq \eta_1$, and $R_{l+1}=  R_l$ if $\rho_l< \eta_1$.
\STATE Update $\nu_{l+1}$ via \eqref{alg:tau-up}.
\STATE Set $l=l+1$.
\ENDWHILE
\end{algorithmic}
\end{algorithm}

\subsection{Extension of SDPDAL}\label{sec:extension}
Our algorithmic framework is compatible with the general SDP problem with inequality constraints, i.e.,
\begin{equation}\label{prob:ineq}
\begin{aligned}
\min_{X\in \mcD}& \quad  f(X)+h(X),\quad \text{s.t.}  \quad\mcA_E(X)=b_E, \quad \mcA_I(X) \leq b_I,
\end{aligned}
\end{equation}
where $\mcA_E:\mbS^n\rightarrow \mbR^{m_E}$ and $\mcA_I:\mbS^n\rightarrow \mbR^{m_I}$ are two linear maps.
\revise{We rewrite \eqref{prob:ineq} as
\blue{\begin{equation}\label{prob:ineq2}
\begin{aligned}
\min_{X\in \mcD,W\in\mathbb{S}^n,s\geq 0}& \;  f(X)+h(W),\; \text{s.t.}  \; \mcA_E(X)=b_E, \mcA_I(X) + s = b_I, X = W.
\end{aligned}
\end{equation}}
Similar to \eqref{sub:xk},  $\Phi_k(X)$ is constructed by minimizing $W$ and $s \geq 0$
simultaneously. Note that the resulting subproblem is still semi-smooth.}
%

SDPDAL can also handle the \emph{nonconvex} SDP with the following form \cite{ijcai2019-157}:
\begin{equation} \label{eqn:epsdp}
\min_{X\in \mcD} \quad  f(X)+h(X)+\lambda E_{\alpha}(X),\quad \text{s.t.}  \quad\mcA(X)=b,
\end{equation}
where $E_{\alpha}(X)$ is a nonconvex entropy penalty with parameter $\alpha$, and $\lambda > 0$ is the regularization parameter.
\revise{ Examples of entropy penalty terms include Tsallis entropy \cite{Tsallis1988} and R\'enyi entropy \cite{Renyi1960}.
Problem \eqref{eqn:epsdp}
is a special case of \eqref{sdp:general} by considering $f(X) + \lambda E_{\alpha}(X)$ as the smooth term. However, we mention that
$E_{\alpha}(X)$ may be a \emph{nonconvex} penalty function thus \eqref{eqn:epsdp} is a nonconvex SDP.
In many real applications, problem \eqref{sdp:general} is lifted from combinatorial optimization problems, which require a rounding procedure to recover the solution to the original problem after solving \eqref{sdp:general}.}
The role of $E_{\alpha}(X)$ is to promote a low-rank (or even rank-one) solution to \eqref{eqn:epsdp} for sufficiently large $\lambda$, which
usually leads to better solutions to the original combinatorial optimization problem.

\revise{
Finally, we briefly explain problems \eqref{prob:rank-constr} and \eqref{prob:hR-fact}. The explicit low-rank constraint in \eqref{prob:rank-constr} allows for a low-rank optimal solution  when a low-rank solution is not admitted in original SDP and our method can be used directly.  
 The constraint $X = R^TR$ in \eqref{prob:hR-fact} is often treated implicitly since $h(R)$ is imposed on $R$ rather than $R^TR$. In fact,  \eqref{prob:hR-fact} is equivalent to 
\begin{equation}
    \min_{R}  \quad f(R^TR) + h(V),\quad \text{s.t.}\quad \mcA(R^TR) = b, R = V, R\in\mcM.
\end{equation}
Consequently, our ALM framework can be applied similar to \eqref{sub:xk} by eliminating the variable $V$. }


\section{Theoretical analysis}\label{sec:thm}
In this section, we present theoretical analysis on 
properties of the nonconvex subproblems of $R$ and convergence analysis of the outer iterations of ALM.


\subsection{Properties of factorized subproblem on Riemannian manifold}\label{sec:3.1}
Although the factorized subproblem \eqref{sub:rk} is non-convex with respect to $R$, the original subproblem \eqref{sub:xk} is convex in $X$. This sheds light in finding global optimum of the non-convex subproblem \eqref{sub:rk}.
 Suppose that \eqref{sub:xk} has an optimal solution $X^*$ which satisfies that $\text{rank}(X^*)\leq p$. Then, the factorized subproblem has the same minimum as the original subproblem \eqref{sub:xk}. Namely, we have
$\min_{R\in \mcM}  \Psi_k(R) = \min_{X\in \mcD} \Phi_k(X)$.
According to the KKT condition, $X^*$ is the minimum of \eqref{sub:xk}
if and only if $ S(X^*)\succeq0$ and $\tr(X^*S(X^*))=0$.  Based on this observation, we characterize the optimality of $R$ for \eqref{sub:rk} in the following proposition. We extend Theorem 7 in \cite{lroot} for SDPs with nonsmooth objective function and general linear equality constraints. 

\begin{proposition} \label{propi:connection}
Suppose that Assumption \ref{assumption1} holds. If $R$ is a rank-deficient local optimum of the factorized subproblem \eqref{sub:rk}, i.e., $\operatorname{rank}(R)<p$, then $R$ is a global optimum of \eqref{sub:rk}.
\end{proposition}

\begin{proof}
If $\text{rank}(R)=p'<p$, let $R=P\hat R$, where $\hat R\in \mbR^{p'\times n}$ is full rank and $P\in \mbR^{p\times p'}$. We can choose $P_\perp\in \revise{ \mbR^{p\times (p-p')}}$ such that
$
P^TP_\perp = 0,\quad P_\perp^TP_\perp=I.
$
For any $V\in \mbR^{(p-p')\times n}$, consider $U=P_\perp V$. Then, $R^TU = \revise{\hat R^T}P^TP_\perp V=0$. Suppose that $\Hess \revise{\Psi_k(R)}$ is positive semidefinite, then, according to Proposition \ref{prop:rgrad},
\[
0\leq \lra{U,\Hess \revise{\Psi_k(R)}[U]} =\revise{2}\tr(USU^T)=\revise{2}\tr(P_\perp VSV^T P_\perp^T) = \revise{2}\tr(VSV^T).
\]
Because the choice of $V$ is arbitrary, $S$ is positive semidefinite. \revise{As $R$ is a local optimum, one can find 
\begin{equation}\nonumber
 0=   \grad  \Psi_k(R)=\nabla  \Psi_k(R)-2\sum_{i=1}^{m_0}u_iRB_i = 2R\nabla \Phi_k(X) - 2R\mathcal{B}^*(u) = 2RS.
\end{equation}
It follows that $\tr(XS) = 0$.} Hence, $X=R^TR$ is the minimizer of \eqref{sub:xk}.
\end{proof}

Let $R^*$ to be the global minima of  \eqref{sub:rk}. Denote $\sigma_i(R)$ as the $i$-th largest singular value of $R$ and let $Q_RQ_R^T$ as the projection matrix on the row space of $R$. Define the distance in $\mbR^{p\times n}$ by
$
\tdist(R_1,R_2)=\min_{ U^TU=I}\|R_1-UR_2\|_F.
$
\revise{ Although the factorized problem is nonconvex, the following proposition
    shows an interesting property of $\Psi_k(R)$ around its global minimum, that
    is, the optimality gap can be bounded in terms of $\|\grad \Psi_k(R)\|_F$
    and $\tdist(R,R^*)$. Moreover, if $R$ is a stationary point satisfying \eqref{equ:distr}, then $R$ is a global minimum of \eqref{sub:rk}. Our result improves Lemma 22 in
\cite{dcffs} by relaxing the condition on $\tdist(R,R^*)$.} 

\begin{proposition}\label{prop:nosaddle}
Suppose that Assumption \ref{assumption1} holds and $R^*$ is of rank $r\leq p$. For $R\in \mcM$ satisfying that
\begin{equation}\label{equ:distr}
\tdist(R,R^*)\leq\beta \sigma_r(R^*)
\end{equation}
with $\beta<1/4$,  we have
\[
\begin{aligned}
 \Psi_k(R)- \Psi_k(R^*)
 \leq &\frac{ (1-\beta)^2(2+\beta)+\beta}{2(1-\beta)^2}\|\grad \revise{\Psi_k(R)}\|_F\tdist(R,R^*).
\end{aligned}
\]
\end{proposition}
\begin{proof}
Let $U^* = \mathop{\arg\min}_{ U^TU=I}\|R-UR^*\|_F$ and $\Delta = R-U^*R^*$. Note that $\tdist(R,R^*)=\|\Delta\|_F$. We
can obtain
\[
\begin{aligned}
&\revise{\Psi_k(R)-\Psi_k(R^*)=}\Phi_k(X)-\Phi_k(X^*)\\
\stepa{\leq}& \lra{\nabla \Phi_k(X),X-X^*} \stepb{=} \lra{\nabla \Phi_k(X)-\mcB^* (u), X-X^*}\\
\stepc{=}&2\lra{S, R^TR-(U^*R^*)^TR}-\lra{S,R^TR+(U^*R^*)^TU^*R^*-2(U^*R^*)^TR}\\
=&2\lra{S,\Delta^TR}-\lra{S,\Delta^T\Delta} \stepd{\leq}  \lra{\grad \Psi_k(R),\Delta}+\left|\lra{S,\Delta^T\Delta}\right|.
\end{aligned}
\]
\revise{Here the step (a) utilizes the convexity of $\Phi_k$ and the step (b) utilizes that \[\lra{\mcB^*(u),X-X^*}=\lra{u,\mcB(X)-\mcB(X^*)}=0.\] The step (c) applies the calculation of gradient in Proposition \ref{prop:rgrad} and the step (d) utilizes that $\grad\Psi_k(R)=RS$.} We first bound $|\lra{S,\Delta^T\Delta}|$:
\begin{equation}\label{bnd:delta}
    \begin{aligned}
\left|\lra{S,\Delta^T\Delta}\right| 
\leq& \|Q_\Delta Q_\Delta^TS\|_2\tr(\Delta^T\Delta)\\
\leq& \revise{ \pp{\|Q_R Q_R^TS\|_2+\|Q_{R^*} Q_{R^*}^TS\|_2 } \tdist(R,R^*)^2},
\end{aligned}
\end{equation}
where the last inequality is due to that the row space of $\Delta$ can be decomposed into the row space of $R$ and the row space of $R^*$. Notice that
\begin{equation}\label{equ:R_star_S}
\|R^*S\|_2=\|R^* Q_{R^*} Q_{R^*}^T S\|_2\geq \| Q_{R^*} Q_{R^*}^TS\|_2 \sigma_r(R^*).
\end{equation}
Using  $\sigma_r(R)\geq \sigma_r(R^*)-\tdist(R,R^*)\geq (1-\beta) \sigma_r(R^*)$ for step (a) below, we have
\begin{equation}\label{eq:RS}
\|RS\|_2=\|\revise{R} Q_{R} Q_{R}^T S\|_2\geq \|Q_{R} Q_{R}^T  S\|_2 \sigma_r(R)\stepa{\geq}(1-\beta) \|Q_{R} Q_{R}^T  S\|_2\sigma_r(R^*).
\end{equation}
Hence,
\begin{equation}\label{bnd:R_star_S}
\begin{aligned}
\|R^*S\|_2\leq& \|RS\|_2+\|\Delta S\|_2\leq \|RS\|_2+\|Q_\Delta Q_\Delta^T S\|_2 \|\Delta\|_2\\
\leq &\|RS\|_2+\pp{ \|Q_{R} Q_{R}^T  S\|_2 + \|Q_{R^*} Q_{R^*}^T  S\|_2 } \beta \sigma_r(R^*)\\
\overset{\eqref{equ:R_star_S}}{\leq} &\|RS\|_2+\frac{\beta}{1-\beta}\|RS\|_2+\beta \|R^*S\|_2 = \frac{1}{1-\beta} \|RS\|_2+\beta\|R^*S\|_2.
\end{aligned}
\end{equation}
This implies that $\|R^*S\|_2\leq(1-\beta)^{-2}\|RS\|_2$. Consequently, we have
\begin{equation}\label{bnd:qrstar_s}
\|Q_{R^*} Q_{R^*}^T  S\|_2\overset{\eqref{equ:R_star_S}}{\leq} \frac{1}{\sigma_r(R^*)} \|R^*S\|_2
\overset{\eqref{bnd:R_star_S}}{\leq} \frac{1}{\sigma_r(R^*)}\frac{1}{(1-\beta)^2}\|RS\|_2.
\end{equation}
Thus, we obtain 
\begin{equation}\label{bnd:delta_final}
\begin{aligned}
\left|\lra{S,\Delta^T\Delta}\right|\overset{\eqref{bnd:delta},  \eqref{bnd:qrstar_s}}{\leq}&\revise{\pp{1+(1-\beta)^{-2}}\frac{\|RS\|_2}{\sigma_r(R^*)}\tdist(R,R^*)^2}\\
\overset{\eqref{equ:distr}}{\leq} &\revise{\frac{\beta(1+(1-\beta)^2)}{(1-\beta)^{2}}}\|RS\|_2\tdist(R,R^*).\\
\end{aligned}
\end{equation}
Hence, it holds 
\[
\begin{aligned}
&\Phi_k(X)-\Phi_k(X^*) \\
\leq & \|\grad  \Psi_k(R)\|_F\tdist(R,R^*)+\revise{\frac{\beta(1+(1-\beta)^2)}{2(1-\beta)^{2}}}\|\grad  \Psi_k(R)\|_F\tdist(R,R^*)\\
=&\revise{\frac{ (1-\beta)^2(2+\beta)+\beta}{2(1-\beta)^2}}\|\grad  \Psi_k(R)\|_F\tdist(R,R^*),
\end{aligned}
\]
\revise{where the inequality utilizes \eqref{bnd:delta_final} and the fact that $\grad\Psi_k(R)=2RS$.}
\end{proof}

According to Proposition \ref{prop:nosaddle} and the nonincreasing of $\Psi_k(R_l)$ with respect to $l$, we can ensure that $\hat R$ is a global minimum of the factorized subproblem \eqref{sub:rk} by requiring $R_0$ to satisfy \eqref{equ:distr} \blue{strictly and the stepsize $s_l$ is chosen appropriately}, \revise{where $\hat{R}$ is the limit point of $\{R_l\}_{l=1}^{\infty}$ with $\lim_{l\rightarrow \infty}\|\grad \Psi_k(R_l) \| = 0 $}. In the following proposition, we show a sufficient condition to bound the optimality gap
$\Psi_k(R)-\inf_{R\in \mcM}\Psi_k(R)$ using the first and second order information of $\Psi_k(R)$ at $R$.
Our results extend Lemma 6 in \cite{tncbm} to SDP with general equality constraints and nonsmooth regularization. 
\begin{proposition}[Optimality gap]\label{prop:gap}
Suppose that Assumption \ref{assumption1} holds and \eqref{sub:xk} has a \revise{global} minimum $X^*$ with rank at most $p$. For $R\in \mcM$ with $\|\grad  \Psi_k(R)\|_F\leq \epsilon_g$ and $S(R^TR)\succeq -\epsilon_HI$, the optimality gap of $ \Psi_k(R)$ at $R$ can be bounded by
\[
 \Psi_k(R)-\min_{R\in \mcM} \Psi_k(R)\leq\sqrt{\operatorname{diam}(\mcD)} \frac{\epsilon_g}{2}+ \operatorname{diam}(\mcD) \epsilon_H,
\]
where $R^*$ is a minimizer of \revise{ \eqref{sub:rk} }and $\operatorname{diam}(\mcD)=\max_{X\in \mcD} \tr(X)$. Moreover, if $I_d\in \operatorname{span}(\{B_i\}_{i=1}^{m_0})$, then the above bound can be refined to
$ \Psi_k(R)-\min_{R\in \mcM} \Psi_k(R)\leq  \operatorname{diam}(\mcD) \epsilon_H.  $
\end{proposition}
\begin{proof}
We first show  for $X,X'\in \mcD$:
\begin{equation}\label{phik:cvx}
 \Phi_k(X)- \Phi_k(X')\leq-\lra{S(X),X-X'}.
\end{equation}
Because $ \revise{\Phi_k(X)}$ is a convex function in $S^n$, we have
\[
 \Phi_k(X)- \Phi_k(X')\leq-\lra{\nabla \Phi_k(X),X-X'}.
\]
On the other hand, for $X,X'\in \mcD$, $\mcB X=\mcB X'=b_0$. 
Hence, we obtain
\[
\begin{aligned}
\lra{S(X), X-X'}
=&\lra{\nabla \Phi_k(X), X-X'}-u^T\mcB(X-X')
=\lra{\nabla \Phi_k(X), X-X'},
\end{aligned}
\]
which proves \eqref{phik:cvx}. Denote $X^*=(R^*)^TR^*$ and $X=R^TR$. Using
\eqref{phik:cvx} gives
\[
\begin{aligned}
 &\Psi_k(R)-\min_{R\in \mcM} \Psi_k(R)
\leq \revise{ -\lra{S(X),X}+\lra{S(X),X^*}}\\
&= \revise{ - \frac{1}{2} \lra{\grad  \Psi_k(R),R}- \lra{S(R^TR),(R^*)^TR^*}}
\leq \sqrt{\operatorname{diam}(\mcD) } \frac{ \epsilon_g}{2} + \operatorname{diam}(\mcD)  \epsilon_H.
\end{aligned}
\]
The last equality comes from $\|R\|_F^2=\tr(R^TR)=\tr(X)\leq \operatorname{diam}(\mcD) $ for $R\in \mcM$. 
If $I_d\in \operatorname{span}(\{B_i\}_{i=1}^{m_0})$, then there exists $\nu$ such that $I_n = \mcB^*(\nu)$, and
\[
\lra{\grad  \Psi_k(R),R} = \lra{ \revise{ R^T \grad  \Psi_k(R)},I_n} = \lra{\revise{ \mcB(R^T \grad  \Psi_k(R))},\nu} = 0.
\]
Hence, we have
$ \Psi_k(R)-\min_{R\in \mcM} \Psi_k(R)\leq \operatorname{diam}(\mcD)  \epsilon_H.
$
This completes the proof.
\end{proof}

\blue{
}
\subsection{Complexity analysis of subproblem}

We present a global convergence analysis of the subproblem as follows. 
We first introduce a lemma. 
\begin{lemma}\label{lem:txu}
For $X\in \mcD$, we have $\|\tilde p\|_F\leq \|U\|_F$ for all $\tilde p \in \p T(X)[U]$.
\end{lemma}
\begin{proof}
We note that $\p T(X)[U]=\operatorname{co}(\p_B T(X)[U])$. For $p\in \p_B
T(X)[U]$, there exists $X_k\in D_T$, $X_k\to X$ such that $p=\lim_{k\to \infty}
T(X_k)[U]$. For $X_k\in \mcD$, we have
\[
\| \revise{\partial T(X_k)}[U] \|_F =  \norm{\lim_{\tau\to 0^+} \frac{T(X_k+\tau U)-T(X_k)}{\tau}}_F\leq\|U\|_F.
\]
As a result, $\|p\|_F = \lim_{k\to \infty}\|\revise{\partial T(X_k)}[U]\|_F\leq \|U\|_F$. Hence, for $\tilde p\in \p  T(X)[U]$, $\|\tilde p\|_F\leq \|U\|_F$.
\end{proof}

We then make the following assumptions. 
\begin{assumption}\label{asmp:a}
We assume that the following statements hold.
\begin{enumerate}[label={\textbf{\Alph*:}},
  ref={\theassumption.\Alph*}]
  \item \label{assump3-A}%
     The Riemannian Hessian and Euclidean gradient are bounded, i.e., there exists $\kappa_H$ and $\kappa_{\Psi_k}$ such that $\|\nabla \Psi_k(R)\|_F \leq \kappa_{\Psi}$ and \revise{ $\| \Hess \Psi_k(R)[U]\|_F \leq \kappa_H \|U\|_F$ holds for any $U\in T_R\mcM$}.
\item \label{assump3-B}%
   There exists $\alpha,\beta > 0$ such that, for all $R\in \mathcal{M}$ and for all $U \in T_{R}\mathcal{M}$,
   \begin{equation}\label{alpha}
     \|\mathcal{R}_R(U) - R\|_F \leq \alpha \|U\|_F, \quad
     \|\mathcal{R}_R(U) - R - U\|_F \leq \beta \|U\|_F^2 .
   \end{equation}
   \item \label{assump3-C}%
   There exist finite constant $M > 0 $ such that, for all $R\in \mathcal{M}$ and $U, V \in \mbR^{n\times p}$, $\|P_{T_R\mathcal{M}}(V) -  P_{T_W\mathcal{M}}(V) \|_F\leq M \|V\|_F\|U\|_F$, where $W = \mathcal{R}_R(U)$.
\end{enumerate}
\end{assumption}
For Assumption \ref{assump3-A}, because $\mcD$ is compact and $\nabla f(X)$ is continuous, $\|\nabla \Phi_k(X)\|_F$ is upper-bounded. This implies that $\|\nabla  \Psi_k(R)\|_F$ is upper-bounded. A sufficent condition of the second part of Assumption \ref{assump3-A} is that there exists some constant $c\geq 0$, for every $\hat H \in \hat \p^2 \Phi_k(X)$, we have
\begin{equation}\label{equ:huc}
\|\hat H[U]\|_F\leq c \|U\|_F.
\end{equation}
The generalized Hessian operator $\hat \p^2 \Phi_k(X)$ satisfies
\[
\hat \p^2 \Phi_k(X)[U] = \nabla^2 f(X)[U]+\sigma_k \mcA^*(\mcA(U))+\p T(X)[U].
\]
On the other hand, it is easy to observe that there exists $c_0\geq 0$ such that
\[
\| \nabla^2 f(X)[U]+\sigma_k \mcA^*(\mcA(U))\|_F\leq c_0\|U\|_F.
\]
Hence, combining with Lemma \ref{lem:txu}, \eqref{equ:huc} holds for $c=c_0+1$.

\revise{Assumption \ref{assump3-B} is standard: a retraction on a compact submanifold always satisfies a second-order boundedness property \cite{grocf}: $\mathcal{R}_R(U) = R + U + \mathcal{O}(\|U\|_F^2)$.}
It is easy to verify that Assumption \ref{assump3-C} is satisfied for specified manifolds in \eqref{specified example}. For example, when $\mcM = \{R\in \mbR^{p\times n}| \|R\|_F=1\}$, we have that
\[
\begin{aligned}
&\|P_{T_R\mathcal{M}}(V) -  P_{T_W\mathcal{M}}(V) \|_F \\
= &\|\tr(R^TV)R-\tr(R^TV)W + \tr(R^TV)W - \tr(W^TV)W\|_F \\
\leq&\|\tr(R^TV)(R-W)\|_F  +  \|\tr((R-W)^TV)W\|_F\\
\leq& 2\|V\|_F\|R-W\|_F \leq 2\alpha\|V\|_F\|U\|_F.
\end{aligned}
\]

We now show that $\nabla \Psi_k(R)$ is Lipschitz continuous with some constant $L_\Psi$. As $\nabla  \Psi(R)=2R\nabla \Phi_k(R^TR)$, it is sufficient to show that $\nabla \Phi_k(X)$ is Lipschitz continuous on $\mcD$. Recall that
\[ \nabla \Phi_k(X) = \nabla f(X)+\sigma_k \mcA^*(\mcA(X)-b-y^k/\sigma_k)+T(X). \]
It is easy to observe that $\nabla f(X)+\sigma_k \mcA^*(\mcA(X)-b-y^k/\sigma_k)$ is Lipschitz continuous. On the other hand,
\[ T(X) = \sigma_k^{-1}\prox_{\sigma_k h^*}(\sigma_k X-Z^k)
=\prox_{(h/\sigma_k)^*}(X-Z^k/\sigma_k).\]
From the non-expansive property of proximal mapping, $T(X)$ is Lipschitz continuous with constant $1$. Hence, $\nabla \Phi(X)$ is Lipschitz continuous on $\mcD$.

\revise{For the rest of this section, we omit the subscript $k$ in $\Psi$ and $\Phi$.}  According to Assumptions \ref{assump3-A} and \ref{assump3-B}, we know that $\Psi$ is a smooth function with retraction $\mcR$, in the sense that, for all $R\in \mathcal{M}$ and all $U \in T_{R}\mathcal{M}$,
\begin{equation}\label{retra-smooth}
 \Psi(\mathcal{R}_R(U)) \leq \Psi(R) + \left<\grad \Psi(R), U \right>  + \frac{\bar{L}_\Psi}{2} \|U\|_F^2,
\end{equation}
where $\bar{L}_\Psi$ is associated with $L_{\Psi},\kappa_{\Psi}, \alpha$ and $\beta$. See \cite[Lemma 2.7]{grocf}.
For simplicity, in what follows we absorb $\revise{s_l}$ into $U_l$ and use $U_l$ to denote $\revise{s_l} U_l$.  We next show that the regularization parameter $\nu_l$ is bounded. 
\begin{lemma}\label{lem:nu_bound}
  Suppose that Assumption \ref{asmp:a} holds. 
  Then, for all $l$, it holds that $\nu_l \leq \nu_{\max}=: \frac{\gamma_2(\bar{L}_\Psi + \kappa_H)}{1-\eta_2 }$.
\end{lemma}
\begin{proof}
Using the definition of $\rho_l$ \eqref{equ:rhol}, the Lipschitz smoothness of $\Psi$  and \eqref{retra-smooth}, it follows that
  \begin{equation}\label{1}
    \begin{split}
       1 - \rho_l  & = 1 - \frac{\Psi(\bar{R}_l) - \Psi(R_l)  }{  m_l( U_l) - \frac{\nu_l}{2}\|U_l\|_F^2}
                 \revise{ \stepa{\leq} } \frac{\Psi(\bar{R}_l) - \Psi(R_l) -  m_l( U_l)  + \frac{\nu_l}{2}\|U_l\|_F^2}{ \frac{\nu_l}{2}\|U_l\|_F^2} \\
                 & \revise{\stepb{\leq}} \frac{\frac{1}{2}(\bar{L}_\Psi + \kappa_H)\|U_l\|_F^2 }{ \frac{\nu_l}{2}\|U_l\|_F^2}  = \frac{(\bar{L}_\Psi + \kappa_H)}{\nu_l},
    \end{split}
  \end{equation}
\revise{ where the step (a) utilizes \eqref{condition} and the step (b) utilizes Assumption \ref{assump-A}.} 
  If $\nu_l \geq \frac{(\bar{L}_\Psi + \kappa_H)}{1-\eta_2 }$, then $\rho_l \geq
  \eta_2$, meaning step $l$ is very successful, and ensures $\nu_{l+1}\leq
  \nu_l$. Thus, we conclude that $\nu_l \leq \frac{\gamma_2(L + \kappa_H)}{1-\eta_2 }$.
\end{proof}

According to Assumption \ref{assump-B}, it is easy to verify that $\Psi$ is bounded from below by a finite constant $\Psi_\mathrm{low}$. The lemma below following from \cite{agarwal2021adaptive} shows that the sum of $U_l$ is bounded above by some constant. 

\begin{lemma}\label{U-bound}
 Suppose that Assumption \ref{assump3-A} holds. Let $\{R_l\}$ and $\{U_l\}$ be the set of iterates and Newton directions generated by Algorithm \ref{alg:ssn_r}.  Then
$ \sum_{k\in \mathcal{S}} \|U_l\|_F^2 \leq \frac{2(\Psi(R_0) - \Psi_{low})}{\eta_1 \nu_{\min}}, $
 where $\mathcal{S}$ is the set of successful iterations.
\end{lemma}
\begin{proof}
If iteration $l$ is successful, then $\rho_l \geq \eta_1 $. According to \cite[Lemma 7]{arnt} and the Armijo condition, it follows that $U_l$ is a descent direction and $ m_l(U_l) <0$. We conclude that
$ \Psi(R_l) - \Psi(R_{l+1}) \geq \eta_1 ( - m_l(U_l) + \frac{\nu_l}{2}\|U_l\|_F^2) \geq \frac{\eta_1 \nu_{\min}}{2}\|U_l\|_F^2.
 $
On the other hand, for unsuccessful iterations, $R_{l+1} = R_l$ and the cost does not change.
Using  Assumption \ref{assump3-A}, a telescoping sum yields
  \[
  \Psi(R_0)  - \Psi_{low} \geq \sum_{k=0}^{\infty} (\Psi(R_l) - \Psi(R_{l+1})) 
  \geq \frac{\eta_1 \nu_{\min}}{2}\sum_{k\in\mathcal{S}} \|U_l\|_F^2.
  \]
\end{proof}

We now state our main complexity result for Algorithm \ref{alg:ssn_r}. The proof follows Theorem 3 in \cite{agarwal2021adaptive}, which assumes that the objective function is twice differentiable on $\mathcal{M}$. However, the Riemannian gradient mapping of $\Psi$ is only semismooth. To upper bound $\|\grad \Psi(R_{l+1})\|_F$, we compare the difference between $\grad  \Psi(R_l)$ and $\grad  \Psi(R_{l+1})$ from different tangent space in the Euclidean sense. Combining with Assumption \ref{assump3-C}, we obtain that $\|\grad \Psi(R_{l+1})\|_F$ is bounded above by $\mathcal{O}(\|U\|_F)$. 
\begin{theorem}\label{theorem}
Suppose that Assumption \ref{asmp:a} holds. For an arbitrary $R_0\in\mathcal{M}$, let $R_0,R_1,\cdots$ be the iterates produced by
Algorithm \ref{alg:ssn_r}. Denote $N = \theta + \nu_{\max} + \alpha L_{\Psi} + \alpha \kappa_{\Psi} M  + \kappa_H$. For any $\epsilon$, the total number of successful iterations $l$ such that
$\|\grad \Psi(R_l)\|_F \leq \epsilon$ is bounded above by
$ K_1(\epsilon) = \frac{2(\Psi(R_0) - \Psi_{low})}{\eta_1 \nu_{\min}} N^2
\frac{1}{\epsilon^2}.$
Furthermore, $\lim_{l\rightarrow \infty}\|\grad \Psi(R_l)\|_F=0$.
\end{theorem}
\begin{proof}
If iteration $l$ is successful, we have $R_{l+1} = \mathcal{R}_{R_l}(U_l)$. The gradient of the model $m_l$ at $U_l$ is given by
  \begin{equation*}
    \begin{split}
       \nabla m_l(U_l)
        =& \grad \Psi(R_{l+1}) + \nu_l U_l
        + \left[\grad \Psi(R_l) + \Hess \Psi(R_l)[U_l] - \grad \Psi(R_{l+1})\right].
    \end{split}
  \end{equation*}
Due to the condition \eqref{condition} and the triangle inequality, we find
  \begin{equation*}
  \begin{aligned}
      \theta \|U_l\|_F
      \geq &\|\nabla m_l(U_l) \|_F
      \geq  \|\grad \Psi(R_{l+1}) \|_F  - \nu_l \|U_l\|_F  \\
    & -\left\|\grad \Psi(R_l) + \Hess\Psi(R_l)[U_l] -  \grad \Psi(R_{l+1})\right\|_F.
  \end{aligned}
  \end{equation*}
  We also note that
  \begin{equation*}
    \begin{split}
      &  \left\|\grad \Psi(R_l) + \Hess\Psi(R_l)[U_l] -  \grad
      \Psi(R_{l+1})\right\|_F\\
      = & \revise{\left\|P_{T_{R_l}}(\nabla \Psi(R_l)) + \Hess\Psi(R_l)[U_l] - 
     P_{T_{R_{l+1}}}( \nabla\Psi(R_{l+1}))\right\|_F} \\
     \stepa{\leq} & \left\|\nabla \Psi(R_{l}) - \nabla \Psi(R_{l+1})\right\|_F + \left\|(P_{T_{R_l}} - P_{T_{R_{l+1}}})\nabla \Psi(R_{l+1})\right\|_F
     + \left\|\Hess\Psi(R_l)[U_l]\right\|_F\\
     \stepb{\leq} & (\alpha L_{\Psi} + \alpha \kappa_{\Psi} M  + \kappa_H ) \|U_l\|_F,
    \end{split}
  \end{equation*}
 \revise{where the step (a) utilizes that the projection on tangent space is non-expansive and the step (b) uses Assumption \ref{asmp:a}.} This indicates that
  \begin{equation}\label{low bound}
    \|\grad \Psi(R_{l+1})\|_F \leq  (\theta + \nu_{\max} + \alpha L_{\Psi} + \alpha \kappa_{\Psi} M  + \kappa_H)\|U_l\|_F.
  \end{equation}
  \revise{Define} a subset of the successful steps based on the tolerance $\epsilon$:
  \[
  \mathcal{S}_{\epsilon}: = \{l: \rho_l \geq \eta_1~ \mbox{and}~ \|\grad \Psi(R_{l+1})\|_F\geq \epsilon\}.
  \]
  For $l\in\mathcal{S}_{\epsilon}$, we can estimate a lower bound of $\|U_l\|_F^2$ using \eqref{low bound} \revise{and the fact} that $\|\grad \Psi(R_{l+1})\|_F > \epsilon$.  Then, calling upon Lemma \ref{U-bound}, we find
  \begin{equation*}
  \frac{2(\Psi(R_0) - \Psi_\mathrm{low})}{\eta_1 \nu_{\min}} \geq \sum_{l\in\mathcal{S}_{\epsilon}} \|U_l\|_F^2 \geq \frac{\epsilon^2}{\revise{ (\theta + \nu_{\max} + \alpha L_{\Psi} + \alpha \kappa_{\Psi} M  + \kappa_H)^2} }|\mathcal{S}_{\epsilon}|.
  \end{equation*}
  The claim regarding limit point can be proved similar to Theorem 3 in \cite{agarwal2021adaptive}.
\end{proof}

According to Lemma 3 in \cite{agarwal2021adaptive}, we know that the number $K$ of successful iterations among $0,\cdots, \bar{l}$ satisfies
  \begin{equation}\label{k-relation}
  \bar{l} \leq \left(1 + \frac{|\log(\gamma_0)|}{\log(\gamma_1)}   \right) K + \frac{1}{\log(\gamma_1)}\log\left(\frac{\nu_{\max}}{\nu_0}\right).
  \end{equation}
Combining Theorem \ref{theorem} with \eqref{k-relation}, we obtain a bound on the total number of iterations for achieving $\| \grad \Psi(R_l) \|_F \leq \epsilon$ as follows.
\begin{corollary}
  Under the assumptions of Theorem \ref{theorem}, Algorithm \ref{alg:ssn_r} produces a point $R_l\in \mathcal{M}$ such that $\|\grad \Psi(R_l)\|_F \leq \epsilon$ in at most
  \begin{equation*}
  \left(1 + \frac{|\log(\gamma_0)|}{\log(\gamma_1)}   \right)  \frac{2(\Psi(R_0) - \Psi_\mathrm{low})}{\eta_1 \nu_{\min}} N^2 \frac{1}{\epsilon^2} + \frac{1}{\log(\gamma_1)}\log\left(\frac{\nu_{\max}}{\nu_0}\right) + 1.
  \end{equation*}
\end{corollary}

\subsection{Asymptotic convergence analysis}
We now analyze the asymptotic convergence rate. We further assume that the following statement holds.
\begin{assumption}\label{ass:b}
There exists $c_0>0$ such that for all $X\in \mcD$ and $u\in \mbR^{m_0}$, we have
\[
u^T\mcB(X\revise{ \mcB^*(u)})\geq c_0\|u\|_2^2.
\]
\end{assumption}
\revise{Assumption \ref{ass:b} implies that the linear operator $\mathcal{L}(u)=\mathcal{B}(X\mathcal{B}^*(u))$ is strictly positive definite and the smallest eigenvalue is strictly lower bounded from zero.} For specified manifolds in \eqref{specified example}, it is easy to verify that Assumption \ref{ass:b} holds. Based on Assumption \ref{ass:b}, we introduce the following lemma. 
\revise{
\begin{lemma}\label{lem:sx_diff}
Suppose that Assumptions \ref{asmp:a} and \ref{ass:b} hold. Let $X_1,X_2\in \mcD$. Suppose that $u_1,u_2$ satisfy $\mcB(X_i(\nabla \Phi(X_i)-\mcB^*(u_i)))=0$. Then, we have 
\[
\|u_1-u_2\|_2\leq \mcO(\|X_1-X_2\|_F).
\]
\end{lemma}}
\begin{proof}
Let $u_{1,2}$ be the unique vector satisfying that $\mcB(X_1(\nabla \Phi(X_2)-\mcB^*(u_{1,2})))=0$. We note that
$ \mcB(X_1(\mcB^*(u_{1,2}-u_1))) = \mcB(X_1(\nabla \Phi(X_2)-\nabla \Phi(X_1)))$.
\revise{Together with the boundedness of $X_1$ implied by Assumption \ref{assump-A}}, we obtain
\[
\begin{aligned}
 c_0 \|u_{1,2}-u_1\|_2^2
&\leq (u_{1,2}-u_1)^T\mcB(X_1(\mcB^*(u_{1,2}-u_1)))\\
&\leq \|u_{1,2}-u_1\|_2 \|\mcB\|_{\text{op}} \|X_1\|_F\|\nabla \Phi(X_2)-\nabla \Phi(X_1)\|_F\\
&=\mcO(\|u_{1,2}-u_1\|_2\|\nabla \Phi(X_2)-\nabla \Phi(X_1)\|_F),
\end{aligned}
\]
which implies $\|u_{1,2}-u_1\|_2=\mcO(\|\nabla \Phi(X_2)-\nabla \Phi(X_1)\|_F)=\mcO(\|X_2-X_1\|_F)$.
On the other hand,
\[
\begin{aligned} \mcB((X_1-X_2)\nabla \Phi(X_2))
&= \mcB(X_1 \mcB^*(u_{1,2}))-\mcB(X_2 \mcB^*(u_{2}))\\
&= \mcB(X_1 \mcB^*(u_{1,2}-u_2))-\mcB((X_2-X_1) \mcB^*(u_{2})).
\end{aligned}
\]
 \revise{Moreover, $S(X)$ is continuous and $\|S(X)\|_F$ is bounded since $u(X)$ is continuous from Assumption \ref{ass:b} and the domain $\mcD$ of $X$ is compact.}
Hence, we have
\[
\begin{aligned}
c_0\|u_{1,2}-u_2\|_2^2
&\leq (u_{1,2}-u_2)^T\mcB(X_1(\mcB^*(u_{1,2}-u_2)))\\
&=\revise{(u_{1,2}-u_2)^T\mcB\Big((X_1-X_2)(\nabla \Phi(X_2)-\mcB^*(u_2))\Big)}\\
&\leq\|u_{1,2}-u_2\|_2 \|\mcB\|_\text{op} \|X_1-X_2\|_F\|S(X_2)\|_F\\
&=\mcO(\|u_{1,2}-u_2\|_2 \|X_1-X_2\|_F),
\end{aligned}
\]
Therefore, $\|u_{1,2}-u_2\|_2=\mcO(\|X_2-X_1\|_F)$.
Consequently, we have
\[
\|u_1-u_2\|_2\leq \|u_{1,2}-u_2\|_2+\|u_{1,2}-u_1\|_2=\mcO(\|X_2-X_1\|_F).
\]
This completes the proof.
\end{proof}
\revise{We note that $S(X_i)=\nabla \Phi(X_i)-\mcB^*(u_i)$. Using the Lipschitz  continuity of $\nabla \Phi$  and  Lemma \ref{lem:sx_diff} gives
$
\|S(X_1)-S(X_2)\|_F\leq \mcO(\|X_1-X_2\|_F).
$}
Therefore, there exists a constant $c_1>0$ such that
\begin{equation}\label{equ:S-c1}
    \|S(X_1)-S(X_2)\|_F\leq c_1 \|X_1-X_2\|_F.
\end{equation} 
We characterize the asymptotic convergence rate as follows. 
\begin{proposition}\label{prop:grad_converge}
Suppose that Assumptions \ref{asmp:a} and \ref{ass:b} hold, $\nabla \Phi$ is semismooth, and $U_l$ satisfies that
\begin{equation}\label{equ:grad-bound}
    \|\grad \Psi(R_l)+\mfH_l(R_l)[U_l]+\nu_lU_l\|_F\leq c_2\|\grad \Psi(R_l)\|_F^{1+\tau},
\end{equation}
where $\mfH_l\in \Hess \Psi(R_l)$, $c_2>0,$ and $0<\tau\leq 1$.
Assume that \revise{\eqref{sub:xk} has a global minimum $X^*=(R^*)^TR^*$ with rank at most $p$} and $R_l$ satisfies
$ \|R_l^TR_l-(R^*)^TR^*\|_F\leq \nu_\text{min}/(4\revise{c_1}). $
Then, we have
\begin{equation}\label{prop3-10-ineq1}
\|\grad \Psi(R_{l+1})\|_F  = \nu_l \|U_l\|_F+\mcO(\|\grad \Psi(R_{l})\|_F^{\revise{1+\tau}}).  
\end{equation}
If we further assume that $\Phi$ is $\sigma$-strongly convex and $\|U_l^TR_l+R_l^TU_l\|_F\geq c_3 \|U_l\|_F$ for some universal constant $c_3>0$, then for $\nu_l<2\sigma c_3^2$, we have the linear convergence rate
\begin{equation}\label{prop3-10-ineq2}
\|\grad \Psi(R_{l+1})\|_F  = \frac{\nu_l}{\nu_l/2+\sigma c_3^2} \|\grad \Psi(R_l)\|_F+\mcO(\|\grad \Psi(R_{l})\|_F^{1+\tau}).
\end{equation}
\end{proposition}
\begin{proof}
As $S((R^*)^TR^*)$ is positive semidefinite, its smallest eigenvalue is non-negative. According to the \revise{Weyl's theorem \cite{weyl1912asymptotische} for singular values}, we have
\[
\begin{aligned}
-\sigma_\text{min}(S(R_l^TR_l))
&\leq \sigma_\text{min}( S((R^*)^TR^*))-\sigma_\text{min}( S(R_l^TR_l))\\
&\leq \| S(R_l^TR_l)-S((R^*)^TR^*))\|_F\\
&\overset{\eqref{equ:S-c1}}{\leq} c_1\|R_l^TR_l-(R^*)^TR^*\|_F\leq \revise{\nu_\text{min}/4} \leq \nu_{l}/4.
\end{aligned}
\]
\revise{By utilizing the equality \eqref{equ:hess_prod} and the fact that $\Phi$ is convex, }we obtain that
\[
\begin{aligned}
\frac{\nu_l}{2}\|U_l\|_F^2&\leq \nu_\text{l}\|U_l\|_F^2+2\lra{U_l^TU_l,S(R_l^TR_l)}
\leq \lra{U_l, \mfH_l(R_l)[U_l]+\nu_l U_l}\\
&\overset{\eqref{equ:grad-bound}}{\leq}\lra{U_l,-\grad\revise{\Psi}(R_l)}+c_2\|\grad\Psi(R_l)\|_F^{1+\tau}\|U_l\|_F\\
&=\mcO(\|U_l\|_F\|\grad\Phi(R_l)\|_F).
\end{aligned}
\]
\revise{
Hence, we have
\begin{equation}\label{equ:ul_norm}
    \|U_l\|_F =\mcO(\revise{\|}\grad\Phi(R_l)\|_F).
\end{equation}
Denote $R_{l+1} - R_l-U_l  = W_l$. According to Assumption \ref{assump3-B},
\begin{equation}\label{equ:wl_norm}
    \|W_l \|_F= \| \mcR_{R_l}(U_l)- R_l-U_l \|_F\leq \beta \|U_l\|_F^2.
\end{equation}
}
For simplicity, we write $u_{l+1}=u(R_{l+1}^TR_{l+1})$ and $u_l=u(R_l^TR_l)$. From Proposition \ref{prop:rgrad}, we can write
\begin{equation}\label{equ:h_l}
    \mfH_l(R_l)[U_l] = 2U_l\nabla \Phi(R_l^TR_l) +2R_l\mfV_l[U_l^TR_l+R_l^TU_l]-2R_l\mcB^*(u_l')-2U_l\mcB^*(u_l),
\end{equation}
where $\mfV_l \in \p^2 \Phi (R_{l}^TR_{l})$ and $u_l'$ satisfies
\[
\mcB\pp{R_l^TR_l\pp{\hat\p^2  \revise{\Phi}(R_l^TR_l)[U_l^TR_l+R_l^TU_l]-\mcB^*(u_l')}}=0.
\]
Because $\nabla \Phi$ is semismooth, by writing $R_{l+1}=R_l+U_l+W_l$, we have
\begin{equation}\label{equ:nabla_phi}
\begin{aligned}
&\nabla  \Phi(R_{l+1}^TR_{l+1})-\mcB^*(u_{l+1})\\
=&\revise{\nabla \Phi(R_{l}^TR_{l})-\mcB^*(u_l)+\mfV_l[(U_l+W_l)^TR_l+R_l^T(U_l+W_l)]}\\ &\revise{+\mcB^*(u_l-u_{l+1})+\mcO(\|U_l\|_F^2)}\\
=&\nabla \Phi(R_{l}^TR_{l})-\mcB^*(u_l)+\mfV_l[U_l^TR_l+R_l^TU_l]+\mcB^*(u_l-u_{l+1})+\mcO(\|\grad\Phi(R_l)\|_F^2),
\end{aligned}
\end{equation}
\revise{where the last equality utilizes \eqref{equ:ul_norm} and \eqref{equ:wl_norm}. 
From Lemma \ref{lem:sx_diff} and Assumption \ref{assump3-B}, we can derive that}
\begin{equation}\label{equ:ul_diff}
\revise{
\begin{aligned}
&\|u_l-u_{l+1}\|_2=\mcO(\|R_l^TR_l-R_{l+1}^TR_{l+1}\|_F) \\
\leq&\mcO  (\| R_l  \|_F + \|R_{l+1}\|_F) \|R_{l+1} - R_{l}\|_F= \mcO(\|U_l\|_F).
\end{aligned}}
\end{equation}
\revise{From Assumption \ref{ass:b}, we note that
\begin{equation}\label{equ:ulp_norm}
    \begin{aligned}
        \|u_l'\|_F &\leq c_0^{-1}\norm{\mcB\pp{R_l^TR_l\hat\p^2  \revise{\Phi}(R_l^TR_l)[U_l^TR_l+R_l^TU_l]}}_F\\
    &\leq \mcO\pp{\|U_l^TR_l+R_l^TU_l\|_F}\leq \mcO\pp{\|U_l\|_F}.    
    \end{aligned}
\end{equation}}
Therefore, we can write
\begin{equation}\label{equ:grad_psi}
\begin{aligned}
&\grad \Psi(R_{l+1})\revise{  \overset{\eqref{equ:r_grad}}{=} }2R_{l+1}(\nabla \Phi(R_{l+1}^TR_{l+1})-\mcB^*(u_{l+1}))\\
\revise{\overset{\eqref{equ:wl_norm}}{=}}& 2R_l (\nabla  \Phi(R_{l+1}^TR_{l+1})-\mcB^*(u_{l+1}))+2U_l (\nabla \Phi(R_{l}^TR_{l})-\mcB^*(u_l)) +\mcO(\|U_l\|_F^2) \\
\revise{\overset{\eqref{equ:nabla_phi}}{=}}& 2R_l (\nabla \Phi(R_{l}^TR_{l})-\mcB^*(u_l)+\nabla^2\Phi(R_{l}^TR_{l})[U_l^TR_l+R_l^TU_l])\\
&+2U_l S(R_l^TR_l)+2R_{l}(\mcB^*(u_l-u_{l+1}))+\mcO(\|U_l\|_F^2) \\
\revise{\overset{\eqref{equ:h_l}}{=}}&\grad \Psi(R_l)+\mfH_l[U_l] +2R_{l+1}\mcB^*(u_l-u_{l+1}+u_l')+\mcO(\|U_l\|_F^2). \\
\end{aligned}
\end{equation}
Because $\grad \Psi(R_{l+1})\in T_{R_{l+1}}\mcM$, we have
\[
\mcB(\grad \Psi(R_{l+1})^TR_{l+1}+R_{l+1}^T\grad \Psi(R_{l+1}))=0.
\]
Therefore, \revise{the definition of the inner products yields}
\begin{equation}\label{equ:grad_product}
\begin{aligned}
&\quad\lra{\grad \Psi(R_{l+1}),2R_{l+1}\mcB^*(u_l-u_{l+1}+u_l')} \\
&=\revise{2}\lra{\grad \Psi(R_{l+1})^TR_{l+1},\mcB^*(u_l-u_{l+1}+u_l')}\\
&\stepa{=}\revise{\lra{\grad \Psi(R_{l+1})^TR_{l+1}+R_{l+1}^T\grad \Psi(R_{l+1}),\mcB^*(u_l-u_{l+1}+u_l')}}\\
&=\revise{(u_l-u_{l+1}+u_l')}^T\mcB(\grad \Psi(R_{l+1})^TR_{l+1}+R_{l+1}^T\grad \Psi(R_{l+1}))=0.
\end{aligned}
\end{equation}
\revise{In the step (a), we utilize that $\mcB^*(u_l-u_{l+1}+u_l')$ is symmetric and the identity $2\lra{A,B}=\lra{A+A^T,B}$ for any $A\in \mbR^{n\times n},B\in \mbS^n$.} In summary, we have the following estimation:
\begin{equation}\label{eq:gradestU}
\begin{aligned}
&\|\grad \Psi(R_{l+1})\|_F^2 \\
\revise{\overset{\eqref{equ:grad_psi}}{=}}&\lra{\grad \Psi(R_{l+1}),\grad \Psi(R_l)+\mfH_l[U_l]}\\
&+\lra{\grad \Psi(R_{l+1}),2R_{l+1}\mcB^*(u_l-u_{l+1}+u_l')}+\mcO(\|\grad \Psi(R_{l+1})\|_F\|U_l\|_F^2)\\
\revise{\overset{\eqref{equ:grad-bound}}{\overset{\eqref{equ:grad_product}}{=}}}&\lra{\grad \Psi(R_{l+1}),-\nu_l U_l}+\mcO(\|\grad \Psi(R_{l+1})\|_F\|\grad \Psi(R_{l})\|_F^{1+\tau})\\
\leq & \nu_l \|\grad \Psi(R_{l+1})\|_F\|U_l\|_F+\mcO(\|\grad \Psi(R_{l+1})\|_F\|\grad \Psi(R_{l})\|_F^{1+\tau}),
\end{aligned}
\end{equation}
\revise{which further yields \eqref{prop3-10-ineq1}.}

If $\Phi$ is $\sigma$-strongly convex and $\|U_l^TR_l+R_l^TU_l\|_F\geq c_3 \|U_l\|_F$, then,
\[
\begin{aligned}
&\lra{U_l, \mfH_l(R_l)[U_l]+\nu_l U_l}\\
\overset{\eqref{equ:hess_prod}}{\geq}&\lra{U_l^TR_l+R_l^TU_l,\mfV_l[U_l^TR_l+R_l^TU_l]}+ \revise{\nu_l}\|U_l\|_F^2+2\lra{U_l^TU_l,S(R_l^TR_l)}\\
\geq& \sigma \|U_l^TR_l+R_l^TU_l\|_F^2+\frac{\nu_l}{2}\|U_l\|_F^2
\geq \pp{\sigma c_3^2+\frac{\nu_l}{2}}\|U_l\|_F^2.
\end{aligned}
\]
Hence, we have
\[
\begin{aligned}
\pp{\sigma c_3^2+\frac{\nu_l}{2}}\|U_l\|_F^2
\leq&\lra{U_l, \mfH_l(R_l)[U_l]+\nu_l U_l}\\
\revise{\overset{\eqref{equ:grad-bound}}{\leq}}& \|U_l\|_F\|\grad\Phi(R_l)\|_F+ \mcO(\|U_l\|_F \|\grad\Phi(R_l)\|_F^{1+\tau}).
\end{aligned}
\]
Combining the above inequality with \eqref{eq:gradestU} completes the proof of \eqref{prop3-10-ineq2}.
\end{proof}
\revise{
\begin{remark}
A sufficient condition for the assumption that $\left\|U_{l}^{\top} R_{l}+R_{l}^{\top} U_{l}\right\|_{F} \geq c_{3}\left\|U_{l}\right\|_{F}$ is that the smallest singular value of $R_l$ is lower bounded by $c_3/2$ for all $R_l$ over the trajectory. 
\blue{We note that the parameter $c_1$ is dependent on the Lipschitz continuity of $\nabla \Phi$ and the manifold $\mcM$. 
The constant $c_1$ may grow with the dimension $n$ and the vanish rate of the term $\nu_{\min}/(4c_1)$ is unclear yet
acceptable by our numerical observation.
The main purpose of Proposition \ref{prop:grad_converge} is to reveal the superlinear convergence rate when $R_l$ is sufficiently close to the optimal solution in the sense that $\|R_l^TR_l - (R^*)^TR^*\|_F$ is small. On the other hand, a local linear convergence rate can be established via Proposition \ref{prop:nosaddle}, whose assumptions are milder compared to the ones in \emph{\cite{dcffs}}.
}
\end{remark}}

\subsection{Convergence analysis of ALM}

Now, we consider the factorized version of the original problem \eqref{sdp:general}:
\begin{equation}\label{pro:fac}
\begin{aligned}
\min_{R\in \mcM, W\in \mbS^n} \quad  f(R^TR)+h(W),\quad \text{s.t.} \quad \mcA (R^TR)=b, \quad R^TR=W.
\end{aligned}
\end{equation}
We say $(R,W)$ satisfies the KKT conditions if there exist Lagrange multipliers $y\in \mbR^m,Z\in \mbS^n$ such that
\begin{equation}\label{kktcond1}
  \begin{aligned}
      \mcA (R^TR)&=b, \quad \hspace{1.3cm}
      R^TR=W, \\
      0&\in \partial h(W) + Z, \quad\quad
      0\in 2R(\nabla f(R^TR) - \mcA^*y -Z) + N_R\mcM,
  \end{aligned}
\end{equation}
where $N_R\mcM$ represents the normal cone of $\mcM$ at $R$. Denote the augmented Lagrangian function associated with \eqref{pro:fac} by
\begin{equation*}
\begin{aligned}
L_\sigma(R,W,y,Z) = &f(R^TR)+h(W)-y^T(\mcA(R^TR)-b)-\lra{Z,R^TR-W}\\
&+\frac{\sigma}{2}\pp{\|\mcA(R^TR)-b\|_2^2+\|R^TR-W\|_F^2}.
\end{aligned}
\end{equation*}
We use the following two stopping criteria in solving the subproblem:
\begin{align}
 \label{equ:stop}\Psi_k(R^{k+1})-\inf_{R\in \mcM}\Psi_k(R) \leq \epsilon_k, \\
 \label{equ:stop1} \|\grad \Psi_k(R^{k+1})\|_F \leq \epsilon_k.
\end{align}
\revise{The stopping criterion \eqref{equ:stop} can be achieved using Propositoin \ref{prop:gap} and techniques in \cite[section 4.3]{WangDengLiuWen2021arxiv}, while the stopping criterion \eqref{equ:stop1} can be achieved using Theorem \ref{theorem}.}

\begin{theorem}\label{thm:alm_kkt}
  Suppose that Assumption \ref{assumption1} holds and the sequence $\{R^k,W^k\}$ generated by Algorithm \ref{alg:ssn_r} satisfies the condition \eqref{equ:stop1} and let $R^*$ and $W^*$ be limit points of $\{R^k\}$ and $\{W^k\}$.
Suppose 
$\lim_{k\to \infty}\epsilon_k=0$ and $\sigma_{k+1}=\sigma_k$. Then, $(R^*,W^*)$ satisfies the KKT conditions \eqref{kktcond1} for the optimization problem.
\end{theorem}
\begin{proof}
 We first examine that Assumption 6.1 in \cite{palmf} holds.  According to the criterion \eqref{equ:stop1}, we can find $\xi^k = \grad \Psi(R^{k+1})$ such that $\|\xi^k\|_F\leq \epsilon_k$, which is equivalent to
  $
  \xi^k - \nabla \Psi(R^{k+1})  \in N_{R^{k+1}}\mcM.$
  Then, there exists $W^{k+1} = \prox_{h/\sigma_k}(X-Z^k/{\sigma_k})$ such that
\[
\begin{aligned}
 \xi^k &\in \nabla_R L_{\sigma_k}(R^{k+1},W^{k+1},y^{k},Z^{k}) + N_{R^{k+1}}\mcM, \\
 0 &\in \partial_W L_{\sigma_k}(R^{k+1},W^{k+1},y^{k},Z^{k}).
\end{aligned}
\]
Then, the result follows from Theorem 6.2 in \cite{palmf}.
\end{proof}

Assume that \eqref{sub:xk} has an optimum with rank smaller than $p$ and we have $I_d\in \text{span}(\{B_i\}_{i=1}^{m_0})$. From Proposition 
\ref{prop:gap}, we only need to compute a second-order stationary point of $\revise{\Psi_k(R)}$ such that
$ S((R^{k+1})^TR^{k+1})\succeq -\frac{\epsilon_k}{\operatorname{diam}(\mcD) } I.
$
Under the condition \eqref{equ:stop}, we establish the global convergence of ALM. 

\begin{theorem}\label{thm:alm_conv}
Suppose that Assumption \ref{assumption1} holds and the sequence $\{R^k,W^k\}$ generated by Algorithm \ref{alg:ssn_r} satisfies the condition \eqref{equ:stop}. Let $X^*$ and $W^*$ be limit points of $\{X^k\}$ and $\{W^k\}$. Suppose that $h(X)$ is continuous or $h(X)=\boldsymbol{1}_{\mcX}(X)$ for some closed set $\mcX$. Assume that $\{\epsilon_k\}$ is bounded.  Then, we have $h(W^*)<\infty$. For any $X\in \mcD$ and $W\in \mbS^n$ satisfying $h(W)<\infty$, we also have
\begin{equation}\label{equ:constraint}
\|\mcA (X^*)-b\|_F^2+\|X^*-W^*\|_F^2\leq \|\mcA (X)-b\|_F^2+\|X-W\|_F^2.
\end{equation}
Moreover, if $\lim_{k\to \infty}\epsilon_k=0$ and $\sigma_{k+1}=\sigma_k$, $(X^*,W^*)$ is a global minimizer of \eqref{prob:p2}.
\end{theorem}
\begin{proof}
We first consider the case where $h(X)$ is continuous. Because $h(X)$ is continuous, $h(W^*)<\infty$. By using the fact that $\mcD$ is a closed set and $\mcD\times \mbS^n$ is closed, the first part of Theorem \ref{thm:alm_conv} follows from the results of Theorem 5.1 in \cite{palmf}. On the other hand, the global convergence of ALM can be established via Theorem 5.2 in \cite{palmf}.  For the case where $h(X)=\boldsymbol{1}_{\mcX}(X)$, we can rewrite \eqref{prob:p2} into
\begin{equation}\label{prob:p2_re}
\begin{aligned}
\min_{X\in \mcD, W\in \mcX}  \quad  f(X), \text{  s.t.}\;  \quad\;\mcA (X)=b, X=W.
\end{aligned}
\end{equation}
This completes the proof.
\end{proof}

We note that $\mcD\times \mcX$ is a closed set. Because $\{W^k\}\in \mcX$ and $\mcX$ is closed, we have $W^*\in \mcX$, i.e., $h(W^*)<\infty$. For any $X\in \mcD$ and $W\in \mcX$, \eqref{equ:constraint} holds according to Theorem 5.1 in \cite{palmf}. Similarly, the global convergence of ALM follows from Theorem 5.2 in \cite{palmf}.

\section{Numerical experiments}\label{sec:num}
In this section, we demonstrate the effectiveness of our proposed algorithm SDPDAL
on a variety of test problems. We implement SDPDAL under MATLAB R2020a. All experiments
are performed on a Linux server with a twelve-core Intel Xeon E5-2680 CPU and 128 GB memory.
The reported time is a wall-clock time in seconds.

We mainly compare SDPDAL with SDPNAL+\footnote{Downloaded from \url{https://blog.nus.edu.sg/mattohkc/softwares/sdpnalplus/} 
} \cite{sdpnal} and its variant
QSDPNAL\footnote{Downloaded from \url{https://blog.nus.edu.sg/mattohkc/softwares/qsdpnal/} 
} \cite{qsdpnal} for most problems. The reasons for not using SDPLR \cite{sdplr} or the algorithms proposed
in \cite{lroot,tncbm,dcffs} include 1) they require special structures of SDPs thus cannot
handle general problems and 2) their performance can
not measure up with SDPDAL during our initial tests. The reason for not comparing
with SSNSDP \cite{assnm} is that it cannot deal with the constraint $X\geq 0$ or
general non-linear objective functions.




When $f$ and $h$ are convex,  the dual problem of SDP \eqref{prob:p2} can be formulated as
\begin{equation}\label{prob:d}
\max_{y\in \mbR^m,u\in \mbR^{m_0},Z\in \mbS^n, S\succeq 0}  y^Tb+u^Tb_0-h^*(-Z)
- f^*(\mcA^*( y)+\mcB^*( u)+Z+S).
\end{equation}
Given the iterate points $(X^k,W^k,y^k,Z^k)$ in Algorithm \ref{alg:alm_ssn} and $u^k := u(X^k), S^k:= S(X^k)$ defined by \eqref{equ:S},  we evaluate the performance of the algorithm via following quantities:
\begin{equation}\label{eqn:crit}
\begin{aligned}
\eta_{p}^k = &\frac{\|\mcA (X^k)-b\|_2}{1+\|b\|_2}, ~~\eta_{Z}^k =  \frac{ \|X^k-W^k\|_F}{1+\|X^k\|_F}, \eta^k_{g} =  \frac{\text{obj}_P - \text{obj}_D}{1+|\text{obj}_P| + |\text{obj}_D|}\\
\eta_{K^*}^k = &\frac{\| \mcP_{S\succeq
0}(-S^k)\|_F}{\|S^k\|_F+1},\eta_{C_1}^k =
\frac{|\lra{X^k,S^k}|}{1+\|X^k\|_F+\|S^k\|_F},
\end{aligned}
\end{equation}
where $\text{obj}_P = f(X^k) + h(X^k)$ and  $\text{obj}_D = (y^k)^Tb+u^Tb_0-h^*(-Z^k) - f^*(\mcA^*( y)+\mcB^*( u)+Z+S)$.
We stop the algorithm if $ \max\{\eta^k_{p},\eta^k_{Z},\eta^k_{K^*},\eta^k_{g},\eta_{C_1}^k\}<\epsilon^\text{tol}, $
where $\epsilon^\text{tol}$ is a given tolerance. Note that $\eta^k_{g}$ is evaluated only when $f$ and $h$ are convex.

Due to page limit, we only report a few summaries of the numerical results. The
detailed tables  can be found in
 \cite{WangDengLiuWen2021arxiv}.
\subsection{Implementation details} 
\revise{The parameters of SDPNAL+ and QSDPNAL are set the same as in \cite{sdpnalp} and \cite{qsdpnal}, respectively. For SDPDAL,  the parameters in Algorithm \ref{alg:alm_ssn} are set:  The ALM step size $\alpha_k$ is chosen from $[1, (1 + \sqrt{5})/2)$,
and $\sigma_k$ is increased by a factor of 1.1 whenever the drop of the infeasibility measure
$\max(\eta_p^k, \eta_Z^k)$ defined in \eqref{eqn:crit}
is not significant. The initial point $R_0$ is randomly selected from the manifold $\mcM$. The parameters in Algorithm \ref{alg:ssn_r} are set as $\eta_1 = 0.01$, $\eta_2 = 0.9$, $\gamma_0 = 0.2$, $\gamma_1 = 1$, $\gamma_2 = 10$, $\nu_{\min} = 10^{-3}$, and $\theta = 0.1$.}

\subsection{Max-cut problems}
\subsubsection{Max-cut SDP with cutting planes}
Given an undirected graph with $n$ nodes, the SDP relaxation of the max-cut problem can be formulated as
\begin{equation}\label{eqn:mc-cut}
    \min  -\frac{1}{4}\lra{C,X}, \;
        \operatorname{s.t.}  \; \diag(X) = e,\; \mathcal{A}(X) \geq -e, \; X\succeq 0,
\end{equation}
where $C$ is the graph Laplacian matrix, and $\mathcal{A}(X) \geq -e$ stands for
a subset of the following cutting planes: $\forall~ 1 \leq i < j < k \leq n,$
\[
\begin{aligned}
    X_{ij} + X_{ik} + X_{jk} &\geq -1,
    X_{ij} - X_{ik} - X_{jk} &\geq -1, \\
    -X_{ij} + X_{ik} - X_{jk} &\geq -1,
    -X_{ij} - X_{ik} + X_{jk} &\geq -1,
\end{aligned}
\]
which are introduced to provide a tighter SDP upper bound. See \cite{Armbruster2012} for more details. \revise{In order to compute $\mathcal{A}(X)$, we only need to collect the related components $X_{ij}$ that appear in $\mathcal{A}$ instead of
forming $X = R^TR$ explicitly.}
In the experiments, we generate the cutting planes in two steps: 1)
 Add an entropy term $\lambda E_{\alpha}(X)$ to \eqref{eqn:mc-cut} and solve the problem without $\mathcal{A}$ to obtain a solution candidate $X^0$.
     2) Choose at most $m$ constraints that are violated the most under $X^0$.
    We choose $m = \lceil \sqrt{n/2} \rceil$ in the following experiments.
A complete \texttt{Gset} 
dataset is tested. 
Since \eqref{eqn:mc-cut} contains inequality constraints only,
we denote
$ \eta_{p} = \frac{\|\mathcal{A}_I(R^TR) - b_I\|_2}{1 + \|b_I\|_2}$, $
    \eta_{C_3} = \frac{|y_I^T(\mathcal{A}_I(R^TR) - b_I)|}{1 +
    \|\mathcal{A}_I(R^TR) - b_I\|_2 + \|y_I\|_2}$.
Other criteria have the same meaning as \eqref{eqn:crit}. 

Table \ref{tab:mc-stat} gives a statistical summary of the comparisons. In the table, ``success'' means
\begin{equation}\label{eqn:eta-max}
\eta_{\max} = \max\{\eta_p, \eta_d, \eta_g, \eta_{K}, \eta_{K^*}, \eta_{C1},
\eta_{C3}\}\le 5\times 10^{-6},
\end{equation}
 where
$ \eta_d = \frac{\|\nabla f(X)-\mcA^*(y)-Z-\mcB^*(u)-S\|_F}{1 + \|\nabla
f(X)\|_F}$ and $\eta_{K} = \frac{\|\mcP_{X\succeq 0}(-X)\|_F}{1 + \|X\|_F}$.
Other quantities follow the definition in \eqref{eqn:crit} with $X = R^TR$.
We remark that $\eta_d$ and $\eta_{K}$ are introduced for a fair comparison
with \mbox{}SDPNAL+. The results returned by SDPDAL always satisfy $\eta_d = \eta_{K} = 0$
due to $X = R^TR \succeq 0$ and the definition of $S$ in \eqref{equ:S}.
The case ``fastest'' means that the CPU time of the algorithm is the least. The case ``fastest under success'' means the fastest algorithm under the success condition. The case ``not slower 10.0 times'' means that the amount of the CPU time of the algorithm is not 10 times slower than the fastest algorithm. The last case corresponds to the ``not slower 10.0 times'' case under the success condition.

\begin{table}[!htb]
\centering
\setlength{\tabcolsep}{5pt}
\caption{A statistic of computational results of SDPDAL and SDPNAL+ for max-cut problems (\texttt{G01--G54})}\label{tab:mc-stat}

\begin{tabular}{|c|cc|cc|}

\hline
& \multicolumn{2}{c|}{SDPDAL}
& \multicolumn{2}{c|}{SDPNAL+} \\
\cline{2-5}
case & number & percentage & number & percentage\\
\hline
success & 54 &  100.0\% & 39 & 72.2\% \\
fastest & 54 &  100.0\% & 0 & 0.0\% \\
fastest under success & 54 &  100.0\% & 0 & 0.0\% \\
not slower 10.0 times & 54 &  100.0\% & 6 & 11.1\% \\
not slower 10.0 times under success & 54 &  100.0\% & 4 & 10.3\% \\
\hline
\end{tabular}
\end{table}

We make the following comments on the results:
1) SDPDAL successfully solves all instances of the \texttt{Gset} dataset.
It takes about 10.5 minutes to solve the largest system \texttt{g81} with $n = 20000$.
We do not report the results of SDPNAL+ for large problems since it fails
to produce a solution to \texttt{g55} ($n=5000$) within $\sim 3$ hours.
 2) SDPNAL+ only solves 72.2\% of the problems successfuly. By inspecting the solver
logs, we find that SDPNAL+ terminates early after detecting no improvements on
the iterations.
 3) According to Table \ref{tab:mc-stat}, SDPDAL is over 10 times faster
than SDPNAL+ in 48 out of 54 instances, mainly because the solutions $X$ are
low-rank ($<0.05 n$) in all cases.

We compare the accuracy and efficiency of SDPDAL with that of SDPNAL+ using the performance profiling method proposed in \cite{dolan2002benchmarking}.
Let $t_{p,s}$ be some performance quantity (e.g. time or accuracy, lower is better) associated with the $s$-th solver on problem $p$.
Then one computes the ratio $r_{p,s}$ between $t_{p,s}$ over the smallest value obtained by $n_s$ solvers on problem $p$, i.e., $r_{p,s} :=\frac{t_{p,s}}{\min\{t_{p,s}: 1\leq s \leq n_s\}}$. For $\tau >0$, the value
$
\pi_{s}(\tau) : = \frac{\text{number of problems where } \log_2(r_{p,s}) \leq \tau}{\text{total number of problems}}
$
indicates that solver $s$ is within a factor $2^\tau \geq 1$ of the performance obtained by the best solver. Then the performance plot is a curve $\pi_s(\tau)$ for each solver $s$ as a function of $\tau$. In Figure \ref{fig:mc-perf}, we show the performance profiles of two criteria: $\eta_{\max}$ defined by \eqref{eqn:eta-max} and CPU time.
In particular, the intercept point of the axis ``ratio of problems'' and the curve in each subfigure is the percentage of the slower/faster one between the two solvers, which is also reflected in the second row of Table \ref{tab:mc-stat}. These figures show that the accuracy and the CPU time of SDPDAL are better than SDPNAL+ on most problems.

\begin{figure}[!htb]
\centering
\subfigure[error $\eta_{\max}$]{
\includegraphics[width=0.45\textwidth]{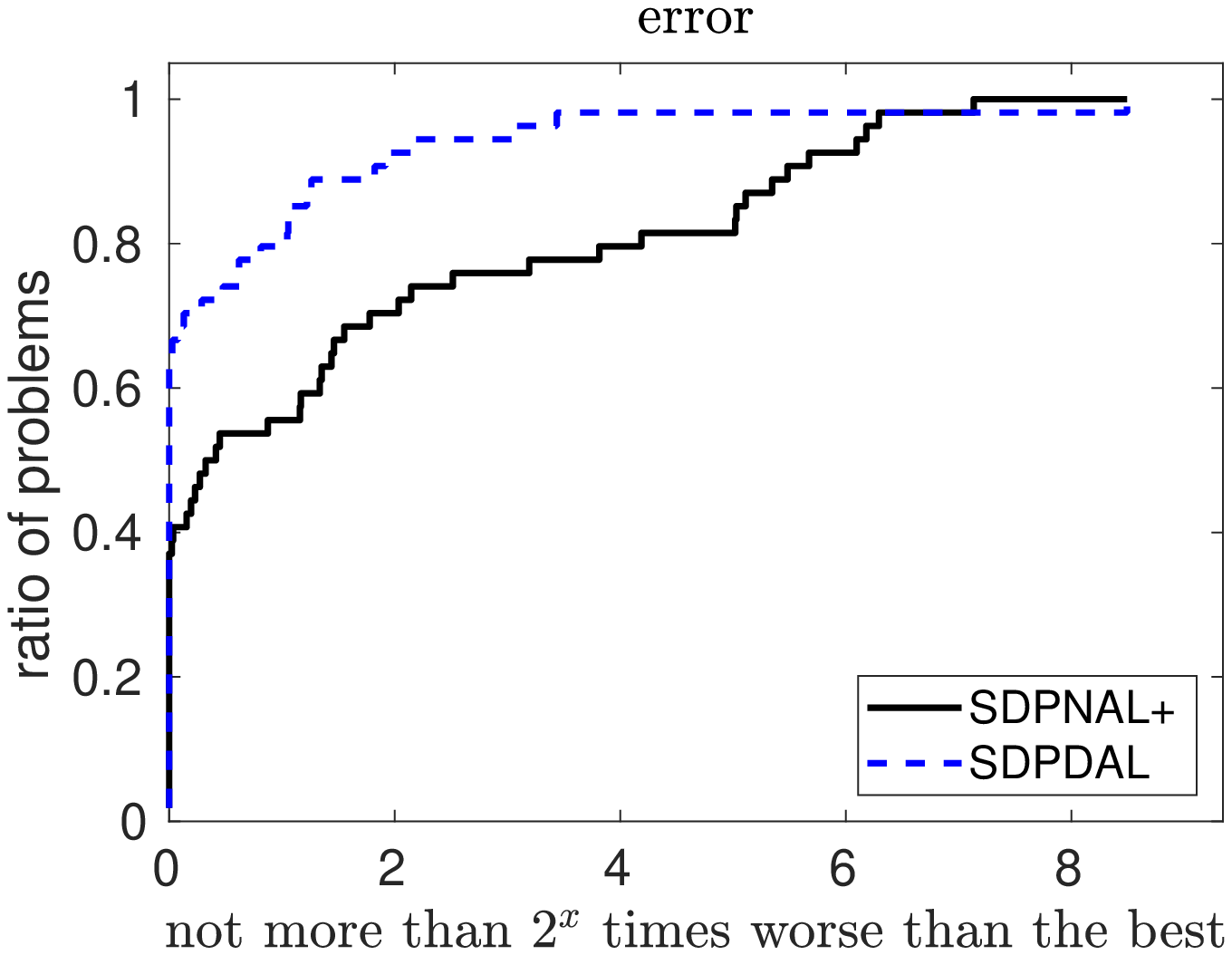}}
\subfigure[CPU]{
\includegraphics[width=0.45\textwidth]{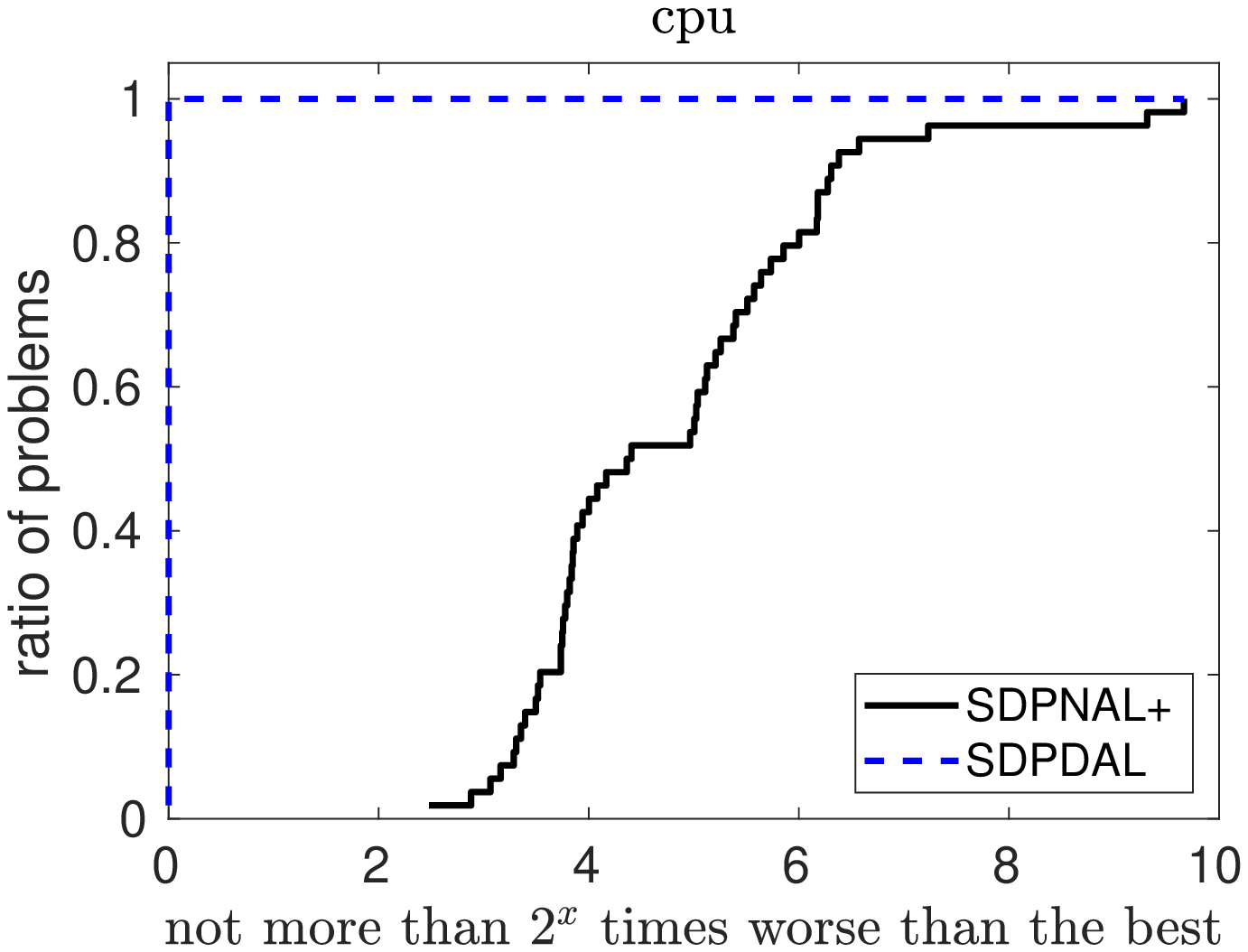}}
\caption{The performance profiles of SDPDAL and SDPNAL+ for max-cut problems \emph{(\texttt{G01--G54})}}\label{fig:mc-perf}
\end{figure}

\subsubsection{Max-cut SDP with entropy-penalization}

%
We compare the performance of Tsallis entropy  ($\alpha = 2$) and R\'enyi entropy
($\alpha = 3$) in \cite{ijcai2019-157} for the formulation \eqref{eqn:epsdp}, where two variants (no cutting planes, with cutting planes) are further tested
for each entropy type. \revise{Due to
the special structure of the entropy functions, computing $X = R^TR$ explicitly is not needed.}
We report the gap between the best known cut value \cite{Gotoeaav2372} (gap\%) defined by
$ \mathrm{gap\%} = 100 \times \frac{\mathrm{best} -
\mathrm{cut}}{\mathrm{best}}$.  A zero ``gap\%'' indicates that the cut value returned by our solver equals to the best known result.
The statistics of the ``gap\%'' of all \texttt{Gset} instances
are demonstrated in Table \ref{tab:mc-ep-stat}, where ``num'' and ``pct''
stand for the number and the percentage of the instances whose ``gap\%''
falls into the corresponding range in the leftmost column.

\begin{table}[!htb]
    \centering
    \caption{Statistics of ``gap\%'' on the \texttt{Gset} dataset.} \label{tab:mc-ep-stat}
    \setlength{\tabcolsep}{1pt}
    \begin{tabular}{|c|cc|cc|cc|cc|cc|}
    \hline
    \multirow{2}{*}{range}
    & \multicolumn{2}{c|}{No entrop}
    & \multicolumn{2}{c|}{Tsallis (no $\mathcal{A}$)}
    & \multicolumn{2}{c|}{Tsallis (with $\mathcal{A}$)}
    & \multicolumn{2}{c|}{R\'enyi (no $\mathcal{A}$)}
    & \multicolumn{2}{c|}{R\'enyi (with $\mathcal{A}$)} \\ \cline{2-11}
    & num & pct & num & pct & num & pct & num & pct & num & pct \\ \hline
        0.00 &    3 &   4.23\% &    3 &   4.23\% &    0 &   0.00\% &    3 &   4.23\% &    2 &   2.82\% \\ \hline
(0.00, 1.00] &    1 &   1.41\% &   19 &  26.76\% &   22 &  30.99\% &   37 &  52.11\% &   39 &  54.93\% \\ \hline
(1.00, 2.00] &    2 &   2.82\% &   25 &  35.21\% &   22 &  30.99\% &   25 &  35.21\% &   22 &  30.99\% \\ \hline
(2.00, 3.00] &   10 &  14.08\% &   13 &  18.31\% &   12 &  16.90\% &    4 &   5.63\% &    8 &  11.27\% \\ \hline
(3.00, 4.00] &   17 &  23.94\% &    6 &   8.45\% &    8 &  11.27\% &    2 &   2.82\% &    0 &   0.00\% \\ \hline
(4.00, 5.00] &    2 &   2.82\% &    3 &   4.23\% &    1 &   1.41\% &    0 &   0.00\% &    0 &   0.00\% \\ \hline
(5.00, 6.00] &    0 &   0.00\% &    2 &   2.82\% &    6 &   8.45\% &    0 &   0.00\% &    0 &   0.00\% \\ \hline
$> 6.00$     &   35 &  50.79\% &    0 &   0.00\% &    0 &   0.00\% &    0 &   0.00\% &    0 &   0.00\% \\ \hline

    \end{tabular}
\end{table}

Below are a few comments on Table \ref{tab:mc-ep-stat}. 
 1) The strong duality may not hold since this SDP problem  is nonconvex. 
 Hence, we ignore $\eta^k_g$ (relative gap) in the stopping rule of SDPDAL.
 2) The optimal gap (gap\%) is smaller than 6\% in all test cases under
 the entropic formulation.
Both Tsallis and R\'enyi entropy are able to
improve the cut value significantly over plain SDP (columns labed by ``No entrop''). 
 3) For R\'enyi entropy, the cut value is slightly better when combined with cutting planes, while no obvious improvement is observed on Tsallis entropy.
We mention that the cutting planes
should be iteratively added to or removed from the SDP problem in order to achieve the best results \cite{biqcrunch}. However, the cutting planes
are fixed in our model  as we only intend to verify
the correctness and speed of SDPDAL. More effective cutting planes
can also be added  for better performance.

\subsection{Relaxation of clustering problems}
The SDP$+$ relaxation of clustering problems (RCP) described in \cite{akcvs} writes
 $$\min\; \lra{-W,X}, \text{ s.t. }Xe=e,\tr(X)=K, X\geq 0, X\succeq 0,$$
where $W$ is the affinity matrix whose entries represent the similarities of the objects in the dataset, $e$ is the vector of ones, and $K$ is the number of clusters. All the datasets we tested are from the UCI Machine Learning Repository, including ``abalone'', ``segment'', ``soybean'' and ``spambase''. For some large
data instances, we only select the first $n$ rows. For example, the original data instance ``spambase'' has 4601 rows, we select the first 1500 rows to obtain the test problem. In our experiment, we set $K = 2,\cdots,11$ respectively. 
\revise{Due to $X\ge 0$, explicitly forming $X = R^TR$ is required.}

The statistics of all examples are shown in Table \ref{tab:stat}.
Apart from the criteria in \eqref{eqn:crit}, we also report
$ \eta_{C_2} = \frac{|\lra{X,Z}|}{1 + \|X\|_F + \|Z\|_F}$
in order to compare with SDPNAL+. Here $Z$ stands for the multiplier associated with $X = W$ in SDPDAL, which is equivalent to the multiplier of $X \geq 0$ in SDPNAL+.
From the table, we can observe that SDPDAL is faster than SDPNAL+ on most examples, for achieving almost the same level of accuracy.   SDPDAL converges fastest on around 82.7\% examples, and it is not 2 times slower than the two other solvers on around 94.5\% examples under the success condition. The corresponding percentage of SDPNAL+ seems to be further smaller than SDPDAL.

\begin{table}[!htb]
\centering
\setlength{\tabcolsep}{5pt}
\caption{A statistic of computational results of SDPDAL and SDPNAL+ on RCP.}\label{tab:stat}

\begin{tabular}{|c|cc|cc|}

\hline
& \multicolumn{2}{c|}{SDPDAL}
& \multicolumn{2}{c|}{SDPNAL+} \\
\cline{2-5}
case & number & percentage & number & percentage\\
\hline
success & 102 &  92.7\% & 71 & 64.5\% \\
fastest & 91 &  82.7\% & 19 & 17.3\% \\
fastest under success & 84 &  82.4\% & 14 & 19.7\% \\
not slower 2.0 times & 104 &  94.5\% & 57 & 51.8\% \\
not slower 2.0 times under success & 97 &  95.1\% & 48 & 67.6\% \\
\hline
\end{tabular}
\end{table}

Figure \ref{fig:perf} shows the performance profile of SDPDAL and SDPNAL+ on
criteria ``error'' and ``CPU''.
These figures again show that the accuracy and the CPU time of SDPDAL are better than SDPNAL+ on most problems.

\begin{figure}[!htb]
\centering
\subfigure[error $\max\{\eta_p, \eta_d, \eta_g, \eta_{K}, \eta_{K^*}, \eta_{C1}\}$]{
\includegraphics[width=0.45\textwidth]{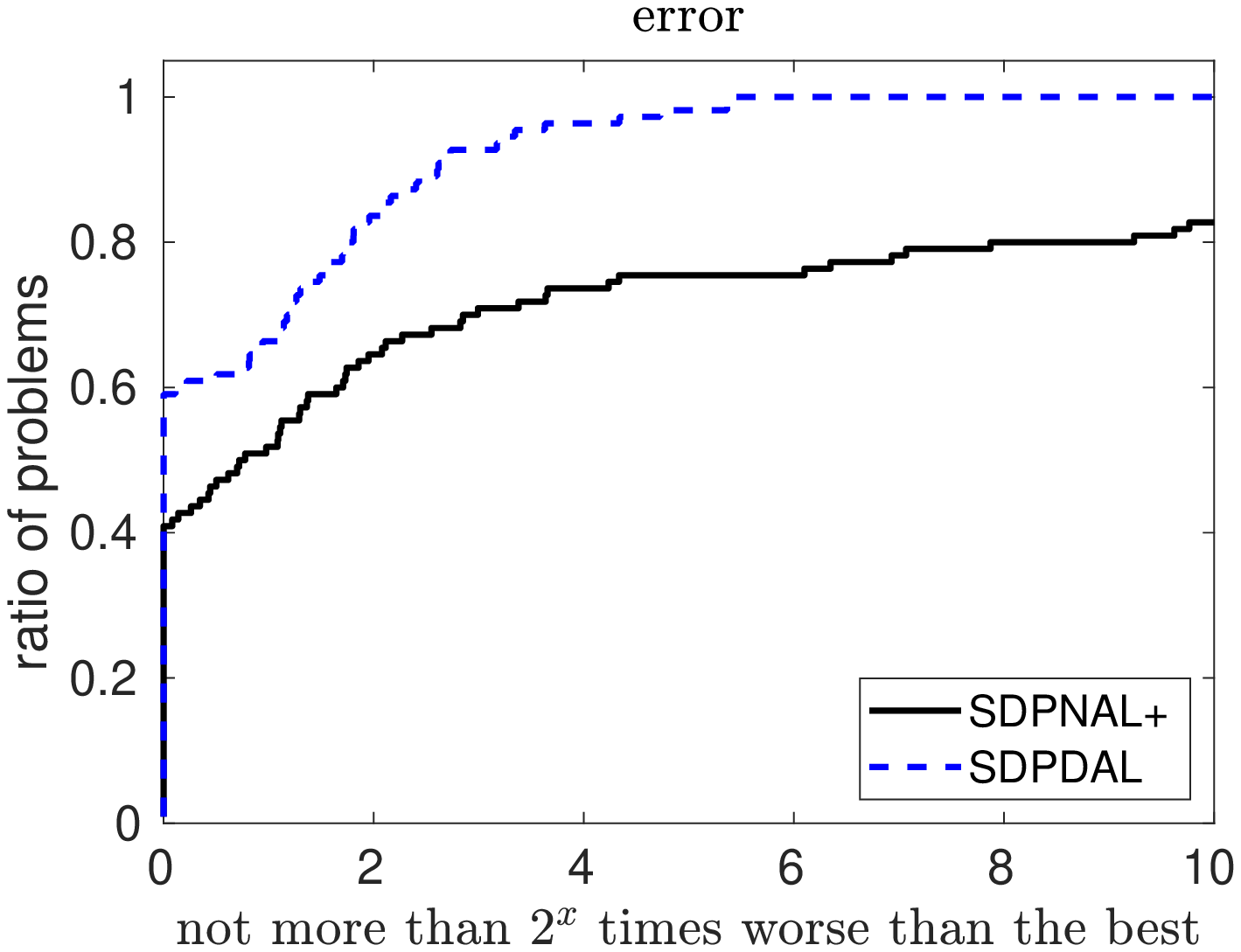}}
\subfigure[CPU]{
\includegraphics[width=0.45\textwidth]{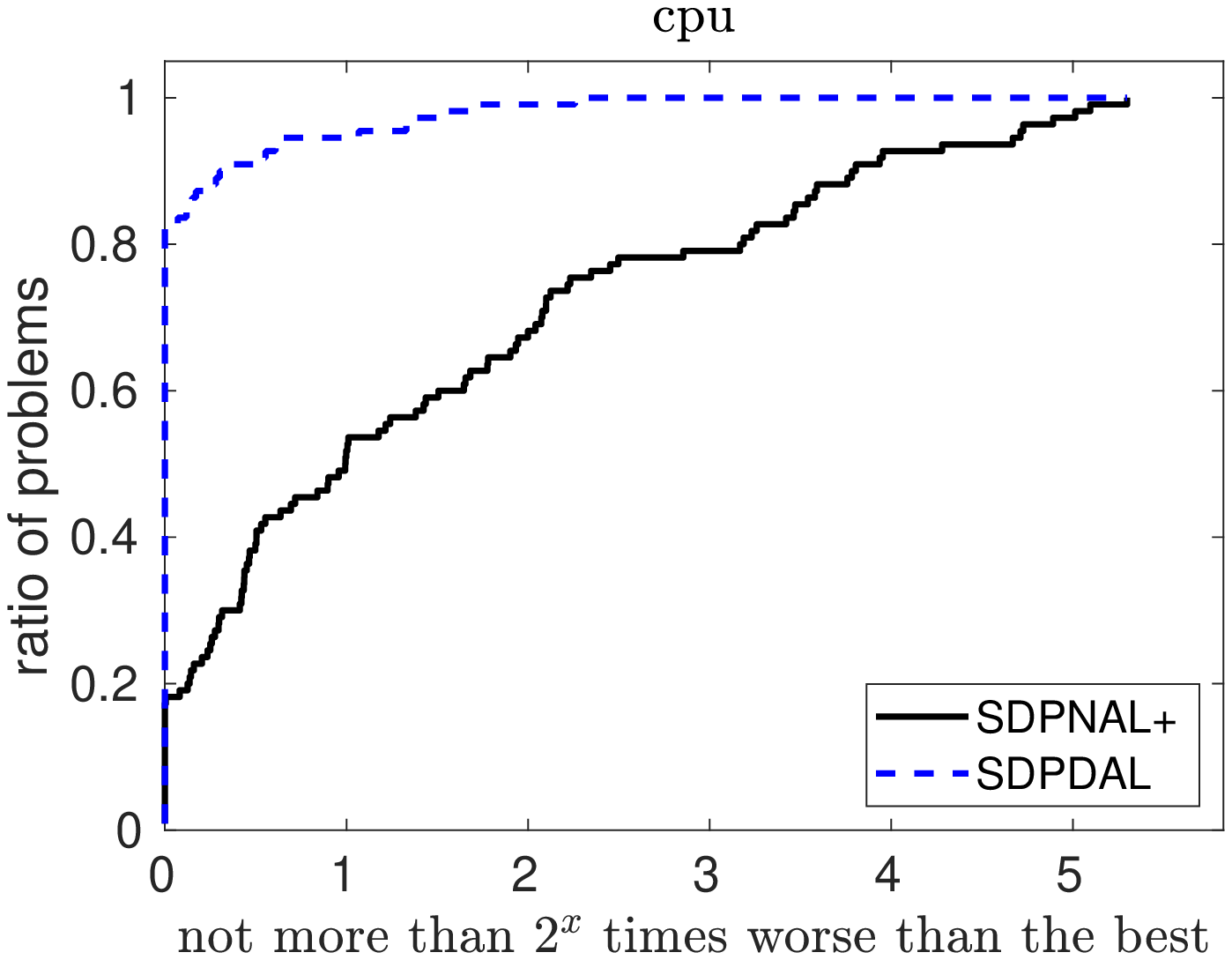}}
\caption{The performance profiles of SDPDAL and SDPNAL+ on RCP}\label{fig:perf}
\end{figure}

\subsection{Theta problems}
Let $G=(V, E)$ be a simple, undirected graph. The Lov{\'a}sz theta SDP \cite{theta} is defined as
\begin{equation} \label{eqn:theta}
        \min \; \lra{-ee^T,X}, \text{ s.t. }  \tr(X) = 1, \; X \succeq 0,\;X_{ij} = 0,\; (i, j) \in E,
\end{equation}
which can be formulated in the form of \eqref{sdp:general} with $f(X) = \lra{-ee^T,X}$, $h(X) = 0$, and $\mathcal{D} = \{X\succeq 0~|~\tr(X) = 1 \}$. \revise{Computing $X = R^TR$ explicitly is not required, since  $f(R^TR)=\lra{Re, Re}=\|Re\|_2^2$ and $X_{ij}$ can be formed from $r_i^Tr_j, (i, j) \in E$, where $r_i$ is the $i$-th column of $R$.} For testing purposes,
we run SDPDAL on the dataset from \cite{sdpnal}, where we drop the instances
\texttt{2dc.\{512,1024,2048\}} since the solutions to these problems are not low rank.
The overall results of all 57 theta instances are demonstrated in Table \ref{tab:theta-stat} and Figure \ref{fig:theta-perf}.

\begin{table}[!htb]
\centering
\setlength{\tabcolsep}{5pt}
\caption{A statistic of computational results of SDPDAL and SDPNAL+ for theta problems}\label{tab:theta-stat}

\begin{tabular}{|c|cc|cc|}

\hline
& \multicolumn{2}{c|}{SDPDAL}
& \multicolumn{2}{c|}{SDPNAL+} \\
\cline{2-5}
case & number & percentage & number & percentage\\
\hline
success & 57 &  100.0\% & 55 & 96.5\% \\
fastest & 53 &  93.0\% & 4 & 7.0\% \\
fastest under success & 53 &  93.0\% & 4 & 7.3\% \\
not slower 1.2 times & 54 &  94.7\% & 11 & 19.3\% \\
not slower 1.2 times under success & 54 &  94.7\% & 11 & 20.0\% \\
\hline
\end{tabular}
\end{table}

We make the following comments on the results:
 1) 
    In summary, SDPDAL is faster than SDPNAL+ on most problems.
    A ten-fold speedup can be observed in several cases.
2) SDPDAL successfully solves all 57 problems, while 2 out of 57 instances
    are partially solved by SDPNAL+. The reason is that SDPNAL+ may allow a larger
    $\eta_g$ even under the most strict stopping rule.
3) In terms of error ($\eta_{\max}$), SDPNAL+ is slightly better than SDPDAL
    by Figure \ref{fig:theta-perf}. According to the detailed output,
    SDPDAL also returns solutions with accuracy under $10^{-6}$ in most cases.
    Thus the error is considered to be at the same level as SDPNAL+.

\begin{figure}[!htb]
\centering
\subfigure[error $\eta_{\max}$]{
\includegraphics[width=0.45\textwidth]{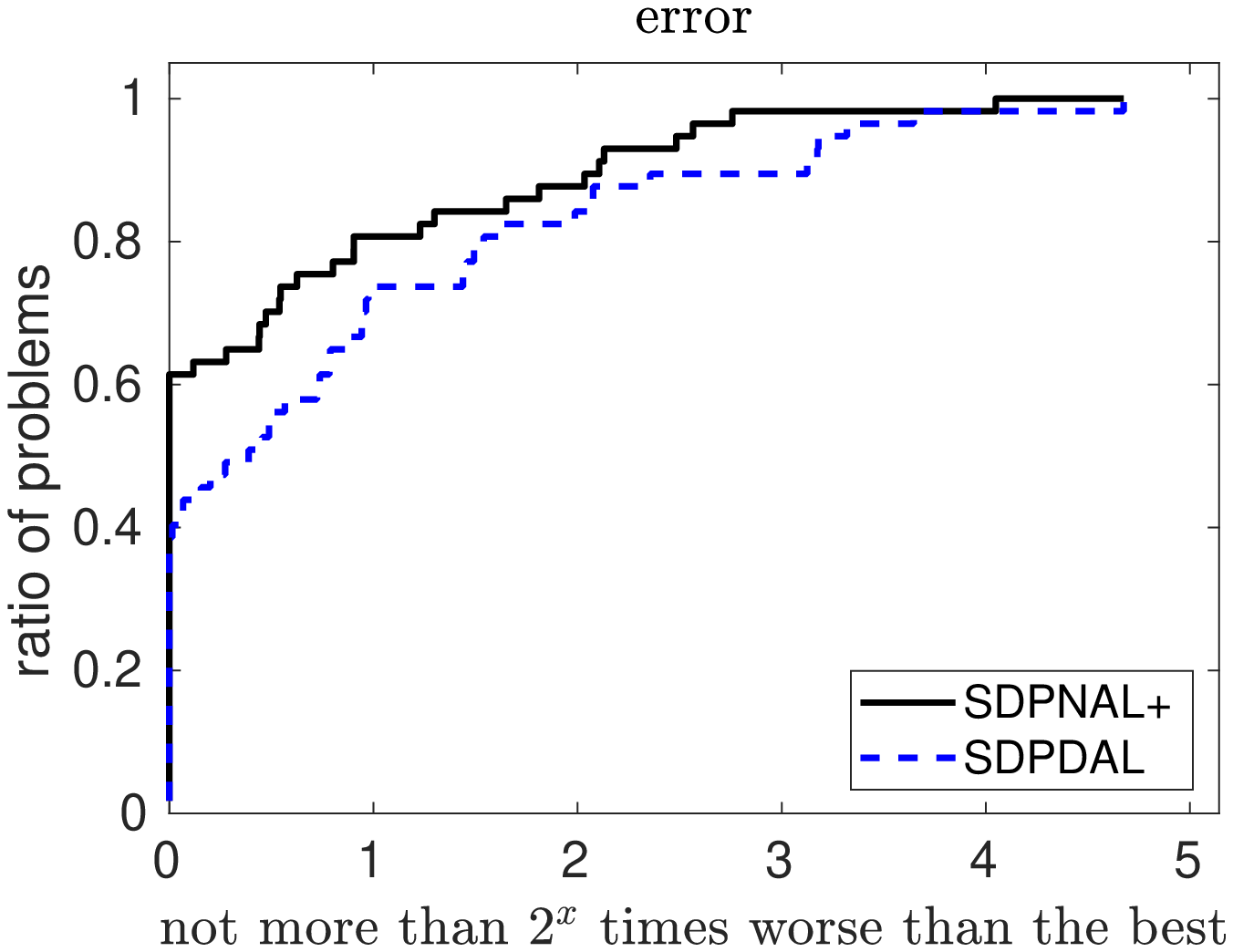}}
\subfigure[CPU]{
\includegraphics[width=0.45\textwidth]{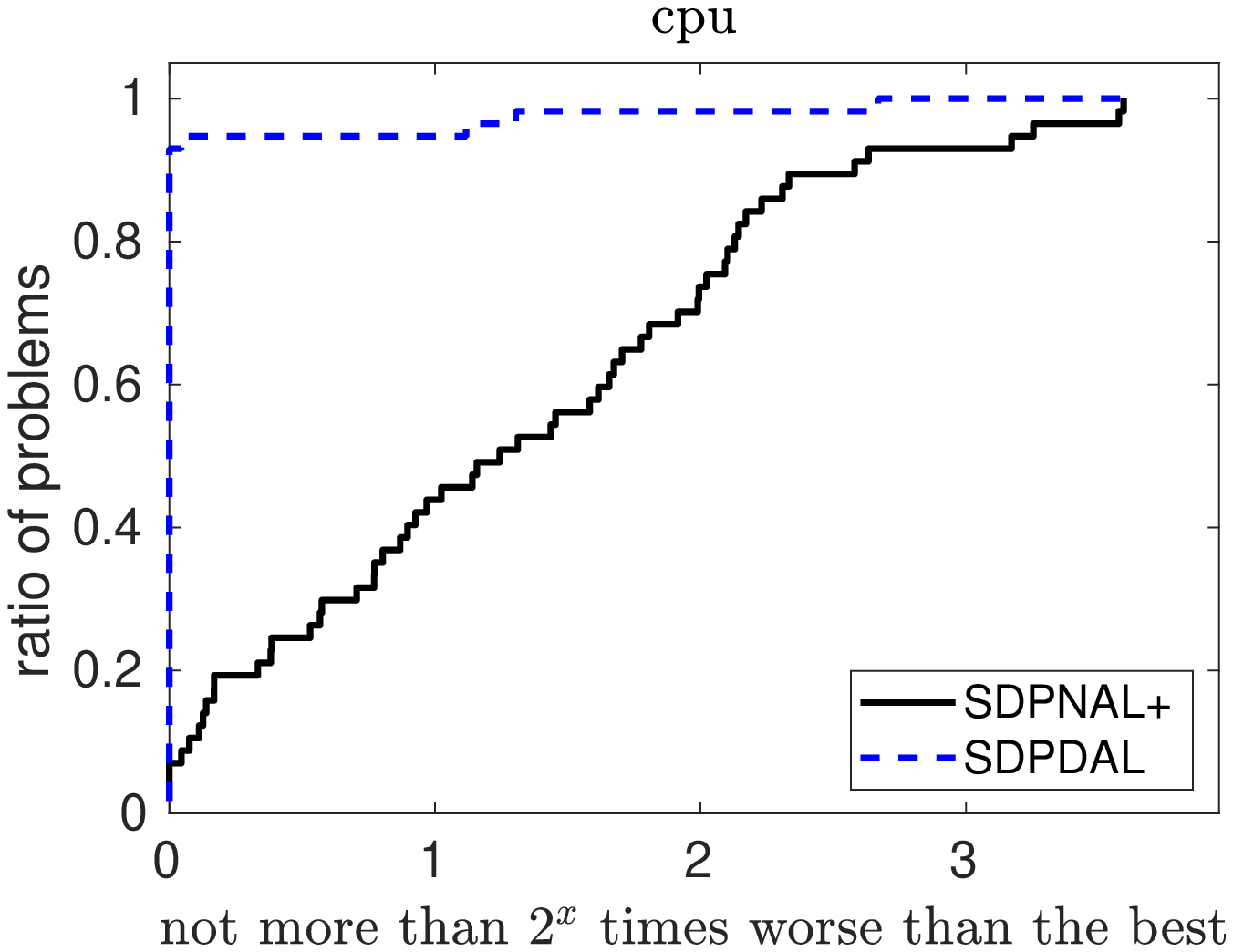}}
\caption{The performance profiles of SDPDAL and SDPNAL+ for theta problems}\label{fig:theta-perf}
\end{figure}

\subsection{Nearest correlation matrix problems (NCM)}
Given a matrix $G\in\mbS^n$, we aim to find the nearest correlation matrix:
\begin{equation}\label{pro:ncm}
  \begin{aligned}
  \min_{X\in\mbS^{n}}\frac{1}{2}\|H\circ (X-G)\|_F^2, \text{ s.t. } X\succeq 0, \text{diag}(X) = e.
  \end{aligned}
\end{equation}
where $H$ is a non-negative weight matrix.  In this experiment, we first take a matrix $\hat{G}$, which is a correlation matrix. Then, we perturb $\hat{G}$ to
$ G = \text{sym}((1-\alpha )\hat{G} + \alpha E)$,
where $\text{sym}(X): = 0.5(X+X^T)$, $\alpha \in (0,1)$ is a given parameter and
$E$ is a low-rank matrix with $E = P_1^TP_2$, where
$P_1,P_2\in\mathbb{R}^{r\times n}$ is randomly generated matrix.  The
construction of the weight matrix $H$ refers to \cite{qsdpnal}. In addition, we
also test the case without $H$, i.e., $H = E$. We choose two datasets from
\cite{li2010inexact}  to generate a matrix $\hat{G}$:
``Leukemia'',``hereditarybc'', and another three instances, namely
``Ross'',``Staunton'' datasets\footnote{See
https://discover.nci.nih.gov/nature2000/natureintromain.jsp}, which come from 60
human tumour cell lines from the National Cancer Institute (NCI).  In our experiment, we set $\alpha = 0.01,0.02,0.05$, respectively.
In addition, we also test the NCM problem with box constraints:
\begin{equation}\label{pro:ncm_box}
  \begin{aligned}
  \min_{X\in\mbS^{n}}\frac{1}{2}\|H\circ (X-G)\|_F^2, \text{s.t.} X\succeq 0, \text{diag}(X) = e, X\in\mcK.
  \end{aligned}
\end{equation}
where $\mathcal{K} = \{X\in\mathbb{S}^n~|~ X\geq l\}$, and  we fix $\alpha = 0.1$ and set $l = -0.3,-0.4,-0.5$, respectively. \revise{Due to the
 elementwise operation ``$\circ$'' and $X \geq l$, computing $X = R^TR$ is required.} 

In Figure \ref{fig:perf_ncm_box}, we show the performance profiles of two criteria ``error'' and ``CPU time'', where error means $\max\{\eta_p,\eta_d, \eta_g\}$.  These figures show that the accuracy and the CPU time of SDPDAL are better than QSDPNAL on most problems. The statistics of all examples are shown in Table \ref{tab:stat_ncm}. Both SDPDAL and QSDPNAL solve all 120 instances, and SDPDAL is faster than QSDPNAL in 116 out of 120 instances.

\begin{figure}[!htb]
\centering
\subfigure[error $\max\{\eta_p,\eta_d, \eta_g\}$]{
\includegraphics[width=0.45\textwidth]{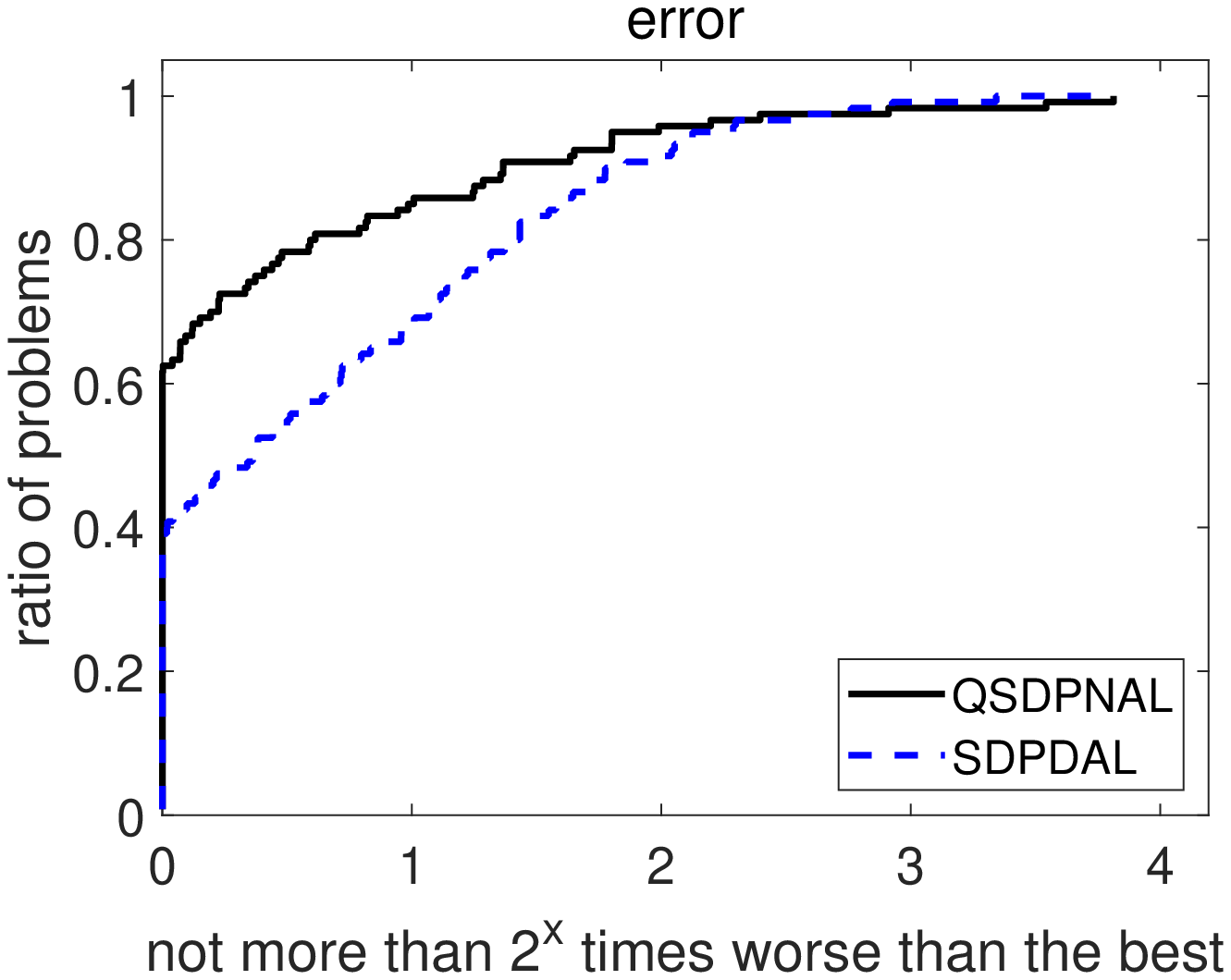}}
\subfigure[CPU]{
\includegraphics[width=0.45\textwidth]{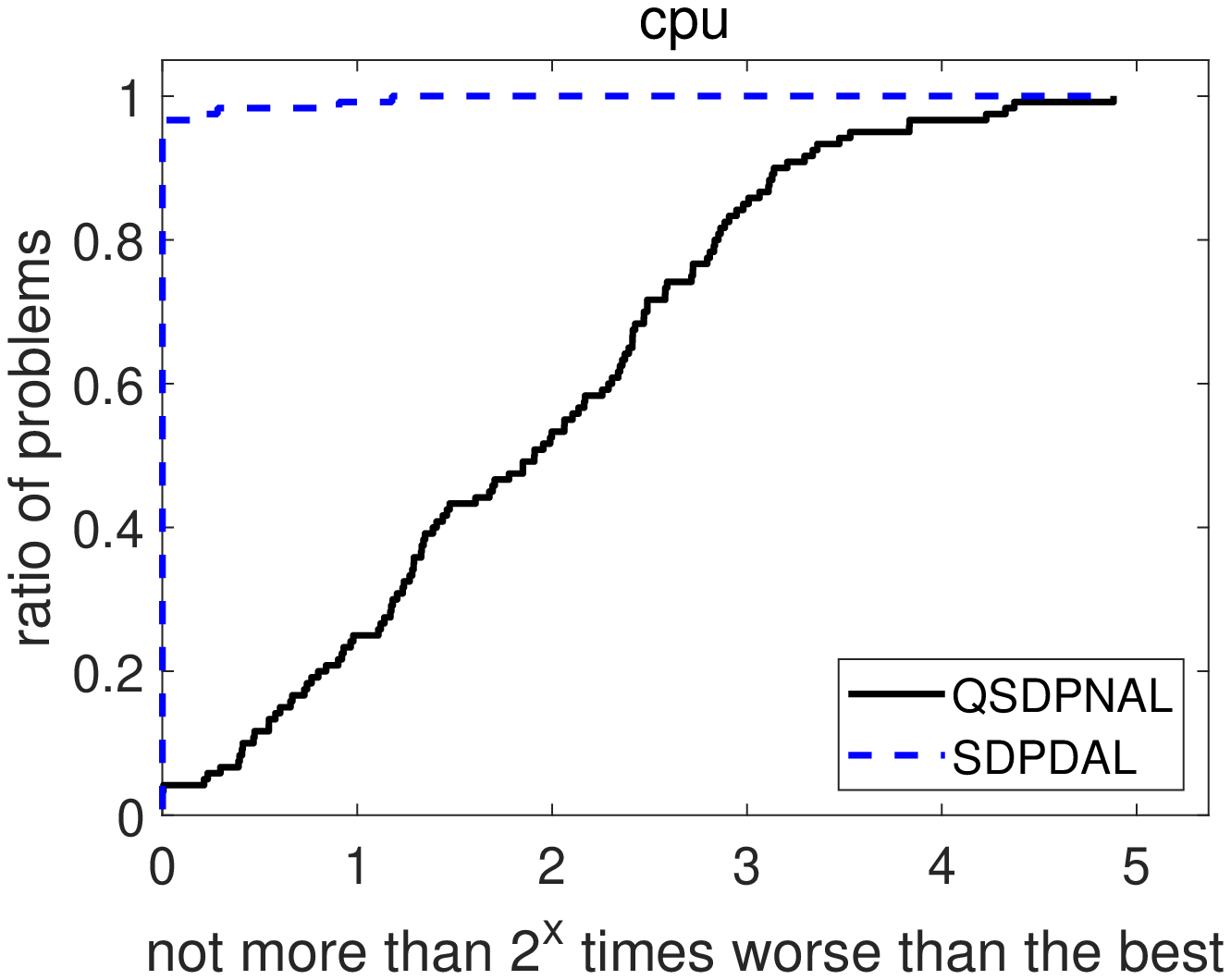}}
\caption{The performance profiles of SDPDAL and QSDPNAL on NCM}\label{fig:perf_ncm_box}
\end{figure}

\begin{table}[!htb]
\centering
\setlength{\tabcolsep}{2.2pt}
\caption{A statistic of computational results of SDPDAL and QSDPNAL on NCM.}\label{tab:stat_ncm}

\begin{tabular}{|c|cc|cc|}

\hline
& \multicolumn{2}{c|}{SDPDAL}
& \multicolumn{2}{c|}{SDPNAL+} \\
\cline{2-5}
case & number & percentage & number & percentage\\
\hline
success & 120 &  100.0\% & 120 & 100.0\% \\
fastest & 116 &  96.7\% & 4 & 3.3\% \\
fastest under success & 116 &  96.7\% & 4 & 3.3\% \\
not slower 1.2 times & 117 &  97.5\% & 7 & 5.8\% \\
not slower 1.2 times under success & 117 &  97.5\% & 7 & 5.8\% \\

\hline
\end{tabular}
\end{table}

\subsection{Sparse PCA with L1 regularization}

The sparse PCA problem for a single component is
 $ \max_{y\in \mbR^n} y^TLy, \; \text{s.t.}\; \|y\|^2=1,\text{card}(y)\leq k.$ 
The function $\text{card}(\cdot)$ refers to the number of non-zero elements. This problem can be expressed as a low-rank SDP:
\begin{equation*}
\min_{X\in \mbS^n} -\lra{L,X}+\lambda \|X\|_1,\quad \text{s.t.}\quad \tr(X)=1,X\succeq 0,
\end{equation*}
where $\|X\|_1=\sum_{ij}|X_{ij}|$ 
\revise{ and explicit computation of  $X = R^TR$ is required.}

For the choice of $L$, we formulate $L$ based on the covariance matrix of real data or use the random example in \cite{spcac}. Namely, $L$ is generated by
$ L = \frac{1}{\|u\|_2^2}uu^T+2VV^T$,
where $u=\bmbm{1,1/2,\dots,1/n}$ and each entry of $V\in \mbR^{n\times n}$ is uniformly chosen from $[0,1]$ at random.
We compare our algorithm with SuperSCS \cite{superscs} \revise{ and DSPCA \cite{d2007direct}.} The results are presented in Table \ref{tab:spca}. In several instances, SuperSCS exceeds the time limit and fails to return a solution, while our algorithm \revise{and DSPCA} can efficiently find the optimal solution. \revise{DSPCA has a better performance compared to the general solver superSCS. Still, SDPDAL compares favorably to 
DSPCA and superSCS.}

\begin{table}[!htb]
\setlength{\tabcolsep}{2pt}
\caption{Computational results of SDPDAL, superCSC and DSPCA on SPCA .}\label{tab:spca}
\centering
\begin{tabular}{|c|c|c|c|c|c|c|c|c|c|c|c|}
  \hline
 &\multicolumn{5}{c|}{SDPDAL}&\multicolumn{3}{c|}{superCSC}&\multicolumn{3}{c|}{DSPCA} \\ \hline
id &  obj & $\eta_g$&$\eta_{K^*}$&$\eta_{C1}$  & time &obj&$\eta_{K}$&time &obj&$\eta_{K}$&time \\ \hline

20news & -3.3+3 & 6.7-11& 2.0-12 & 2.0-12 & 0.8&-3.3+3&2.5-9 &486&  -3.3+3 & 3.8-10 & 3.8\\\hline
bibtex & -1.8+4 & 2.1-9& 1.2-11 & 1.2-11 & 76.6&-3.1+2&1.6-2 &2885& -1.8+4 & 2.8-10 & 1021\\\hline
cancer & -1.8+4 & 1.1-9& 5.5-12 & 5.5-12 & 45.9&-3.1+2&4.8-3 &3567& -1.8+4 & 2.7-9 & 766.5\\\hline
delicious & -7.5+4 & 1.7-10& 2.6-12 & 2.6-12 & 2.9&-7.5+4&5.0-9 &1953& -7.5+4 & 5.7-9 & 39.3\\\hline
dna & -1.8+3 & 1.1-9& 1.7-13 & 1.2-13 & 0.3&-1.8+3&7.9-9 &1528& -1.8+3 & 7.9-12 & 4.8\\\hline
gisette & -3.9+5 & 6.7-10& 2.5-12& 2.5-12 & 1190&//&// &//& -3.9+5 & 3.5-13 & 18678\\\hline
madelon & -9.5+7 & 5.0-13& 5.9-15 & 5.9-15 & 16.7&-9.5+7&5.7-10 &956& -9.5+7 & 4.0-14 & 54.2\\\hline
protein & -3.0+3 & 3.5-9& 3.5-11 & 3.5-11 & 3.7&-3.0+3&1.2-6 &3866& -3.0+3 & 1.2-14 & 33.1\\\hline
rand2048 & -2.1+6 & 3.9-16& 7.4-18& 1.5-18 & 2.3&//&// &//& -2.1+6 & 2.7-13 & 262.9\\\hline
rand4096 & -8.4+6 & 1.7-16& 8.1-18& 8.2-18 & 73.4&//&// &//&-8.4+6&2.4-13 &492.1\\\hline

\end{tabular}
\end{table}

\section{Conclusions}
In this paper, we proposed a decomposition method based augmented Lagrangian framework for solving low-rank semidefinite programming problems, possibly  with nonlinear  objective functions, nonsmooth regularization, and general linear equality/inequality constraints.
{The key strategy is to separate the structured constraints for matrix
    factorization and deal with other constraints using ALM and splitting. Each ALM subproblem
can be efficiently solved by a semismooth Newton method on \revise{a} manifold.
Theoretically, we analyze sufficient conditions for the global optimality of the factorized subproblem and establish convergence analysis for both the Riemannian subproblem and the augmented Lagrangian method. 
Numerical comparisons on various test problems show that our method compares favorably with other algorithms, especially for large problems. 
 Our algorithmic framework is quite general and can be extended to other scenarios  as long as a low-rank solution is admitted and the manifold structure is simple.

}

\section*{Acknowledgements} The authors are grateful to Prof. Samuel Burer and
three  anonymous referees for their valuable comments and suggestions.

\bibliographystyle{siamplain}
\bibliography{SDP}
\end{document}